\documentclass[onefignum,onetabnum]{siamonline190516}

\usepackage{lipsum}
\usepackage{amsfonts}
\usepackage{graphicx}
\usepackage{epstopdf}
\usepackage{algorithmic}
\ifpdf
  \DeclareGraphicsExtensions{.eps,.pdf,.png,.jpg}
\else
  \DeclareGraphicsExtensions{.eps}
\fi

\usepackage{enumitem}
\setlist[enumerate]{leftmargin=.5in}
\setlist[itemize]{leftmargin=.5in}

\newsiamremark{remark}{Remark}
\newsiamremark{problem}{Problem}
\newsiamremark{hypothesis}{Hypothesis}
\crefname{hypothesis}{Hypothesis}{Hypotheses}
\newsiamthm{claim}{Claim}

\headers{Dynamics of the Suarez-Schopf delayed oscillator}{M. M. Anikushin, A. O. Romanov}

\title{Hidden and unstable periodic orbits as a result of homoclinic bifurcations in the Suarez-Schopf delayed oscillator and the irregularity of ENSO\thanks{Submitted to the editors DATE.
\funding{The reported study was funded by the Russian Science Foundation (Project 22-11-00172).}}}

\author{Mikhail Anikushin\thanks{Department of
		Applied Cybernetics, Faculty of Mathematics and Mechanics,
		St. Petersburg University, 28 Universitetskiy prospekt, Peterhof, 198504, Russia
  (\email{demolishka@gmail.com}).}, 
  \and Andrey Romanov\thanks{Department of
  	Applied Cybernetics, Faculty of Mathematics and Mechanics,
  	St. Petersburg University, 28 Universitetskiy prospekt, Peterhof, 198504, Russia
  (\email{romanov.andrey.twai@gmail.com}).}
}

\usepackage{amsopn}

\ifpdf
\hypersetup{
  pdftitle={Hidden and unstable periodic orbits as a result of homoclinic bifurcations in the Suarez-Schopf delayed oscillator and the irregularity of ENSO},
  pdfauthor={M.M. Anikushin, A.O. Romanov}
}
\fi

\begin{document}

\maketitle

% REQUIRED
\begin{abstract}
  We revisit the classical Suarez-Schopf delayed oscillator. Special attention is paid to the region of linear stability in the space of parameters. By means of the theory of inertial manifolds developed in our adjacent papers, we provide analytical-numerical evidence for the existence of two-dimensional inertial manifolds in the model. This allows to suggest a complete qualitative description of the dynamics in the region of linear stability. We show that there are two subregions corresponding to the existence of hidden or self-excited periodic orbits. These subregions must be separated by a curve on which homoclinic ``figure eights'', bifurcating into a single one or a pair of unstable periodic orbits, should exist. We relate the observed hidden oscillations and homoclinics to the irregularity theories of ENSO and provide numerical evidence that chaotic behavior may appear if a small periodic forcing is applied to the model. We also use parameters from the Suarez-Schopf model to discover hidden and self-excited asynchronous periodic regimes in a ring array of coupled lossless transmission lines studied by J.~Wu and H.~Xia.
\end{abstract}

% REQUIRED

% REQUIRED

\section{Introduction}
In the present paper we revisit the classical delayed oscillator proposed by M.J.~Suarez and P.S.~Schopf in \cite{Suarez1988} as a model for El Ni\~{n}o--Southern Oscillation (ENSO). It is given by the scalar delay equation
\begin{equation}
	\label{EQ: ElNinoSSmodel}
	\dot{x}(t) = x(t) - \alpha x(t-\tau) - x^{3}(t),
\end{equation}
where $\alpha \in (0,1)$ and $\tau>0$ are dimensionless parameters. For these parameters $\phi^{0}(\cdot) \equiv 0$ is always a stationary state with a one-dimensional unstable manifold. Moreover, there also exists a pair of symmetric stationary states $\phi^{+}(\cdot) \equiv \sqrt{1-\alpha}$ and $\phi^{-}(\cdot) \equiv -\sqrt{1-\alpha}$. Below we pay special attention to the region in the space of parameters $(\tau,\alpha)$, where $\phi^{+}$ and $\phi^{-}$ are linearly stable. We call it the \textit{region of linear stability}.

Note that neither M.J.~Suarez and P.S.~Schopf \cite{Suarez1988} nor I.~Boutle, R.H.S.~Taylor and R.A.~R\"{o}mer \cite{Boutleet2007ElNino}, who conducted independent simulations of the model, found oscillations in the region of linear stability. Apparently, they were misled by the standard local linear analysis and the intuition that the model should describe the linear instability route to periodicity (see Chapter 6 in the monograph of M.J.~McPhaden, A.~Santoso and W.~Cai \cite{McPhadenBook2020} for a discussion). A more careful analysis shows that there may exist self-excited periodic orbits for parameters close to the \textit{neutral curve} (the term used in \cite{Suarez1988}), where the symmetric equilibria lose their stability and the subcritical Hopf bifurcation occurs. As we approach the region of true stability (with a gradient-like behavior), these self-excited oscillations become hidden (see Appendix \ref{APP: LocalizationHiddenAttrators} for a brief introduction to the theory of hidden attractors) and then disappear. There is also a variety of periods for these oscillations, which may agree with real ENSO events. These regions are schematically shown in Fig. \ref{FIG: SSHiddenCurves} (see Section \ref{SEC: SSmodelHidden} for details).
\begin{figure}
	\centering
	\includegraphics[width=1.\linewidth]{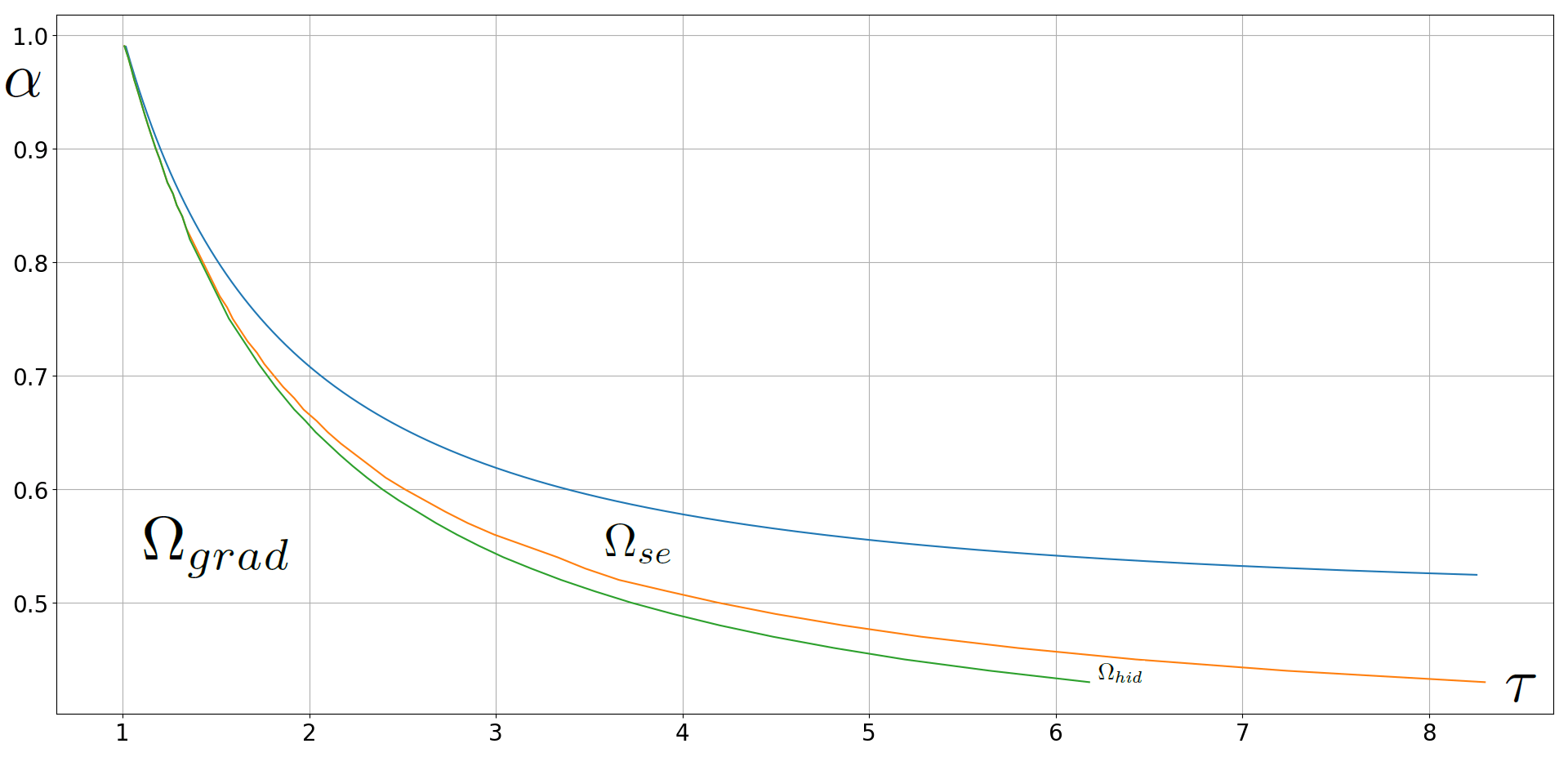}	
	\caption{The lower hidden curve (green), the upper hidden curve (orange) and the neutral curve (blue). In the region $\Omega_{se}$ between the orange and the blue curves, only phase portraits with self-excited periodic orbits corresponding to Fig. \ref{fig: SSPPunstableTwo} are observed. In the region $\Omega_{hid}$ between the green and the orange curves, only hidden periodic orbits with phase portraits corresponding to Fig. \ref{fig: SSPPunstable} are observed. The region $\Omega_{grad}$ under the green curve is expected to be nonoscillatory with a gradient-like behavior. At the points from the upper hidden curve phase portraits with homoclinic ``figure eights'' corresponding to Fig. \ref{fig: SSPPHomo} are expected.}
	\label{FIG: SSHiddenCurves}
\end{figure}

We are grateful to the anonymous referee for suggesting that the unique zero stationary point for $\alpha=\tau=1$ is a double-zero singularity with reflection symmetry, which unfolding is discussed in the monograph of Yu.A.~Kuznetsov \cite{KuznetsovElemtsAppliedBF1998} (see Fig. 9.10 therein). Thus, using the normal form theory (see, for example, the monograph of S.~Guo and J.~Wu \cite{GuoWu2013BifTh}) one can rigorously prove the existence of the curves from Fig. \ref{FIG: SSHiddenCurves} and the corresponding phase portraits in a neighborhood of the point $(\tau,\alpha)=(1,1)$.

Moreover, the same referee drew our attention to the recent paper of S.K.J.~Falkena et al. \cite{FalkenaQuinnMZform2019}, where using the MATLAB toolbox DDE-BIFTOOL a bifurcation diagram for \eqref{EQ: ElNinoSSmodel} was computed. In particular, Fig. 4 in \cite{FalkenaQuinnMZform2019} contains the upper hidden curve from our Fig. \ref{FIG: SSHiddenCurves}, but it did not contain the lower hidden curve and did not mention the presence of hidden oscillations in the model. Of course, this was beyond the interest of \cite{FalkenaQuinnMZform2019}, but it also shows that the standard software for numerical analysis does not distinguish between hidden and self-excited oscillations even in the simplest models.

In this paper, we focus on the region $\Omega_{hid}$, where hidden periodic orbits exist, and place emphasis on its role in understanding the ENSO phenomenon and other oscillators.

\begin{figure}
	\centering
	\includegraphics[width=1\linewidth]{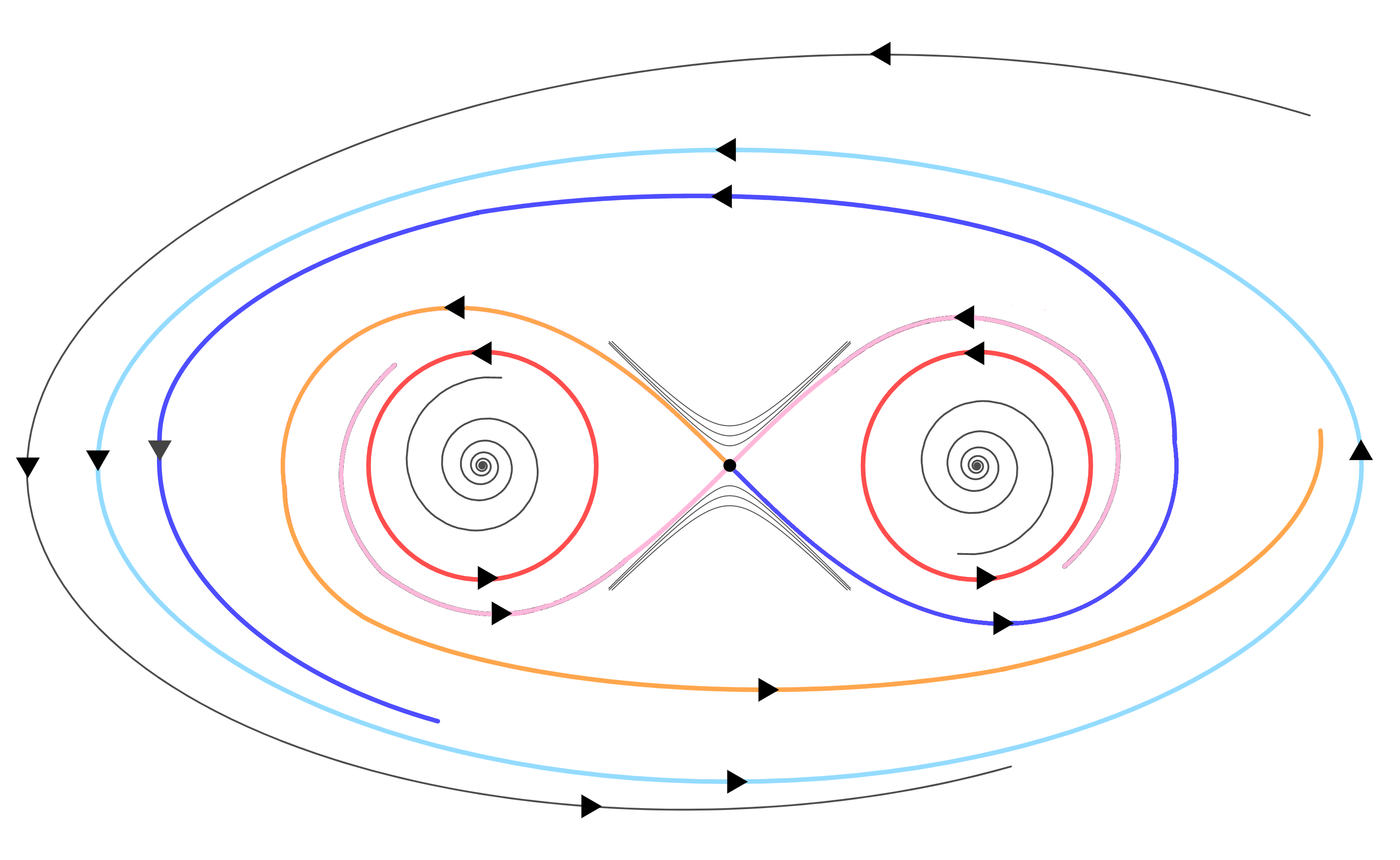}
	\caption{A self-excited attracting periodic orbit (cyan) coexists with two unstable periodic orbits (red) born after a homoclinic bifurcation. The unstable separatrices (blue and orange) tend to the attracting periodic orbit. The stable separatrices (pink) tend to the unstable periodic orbits in the negative direction of time.}
	\label{fig: SSPPunstableTwo}
\end{figure}
\begin{figure}
	\centering
	\includegraphics[width=1\linewidth]{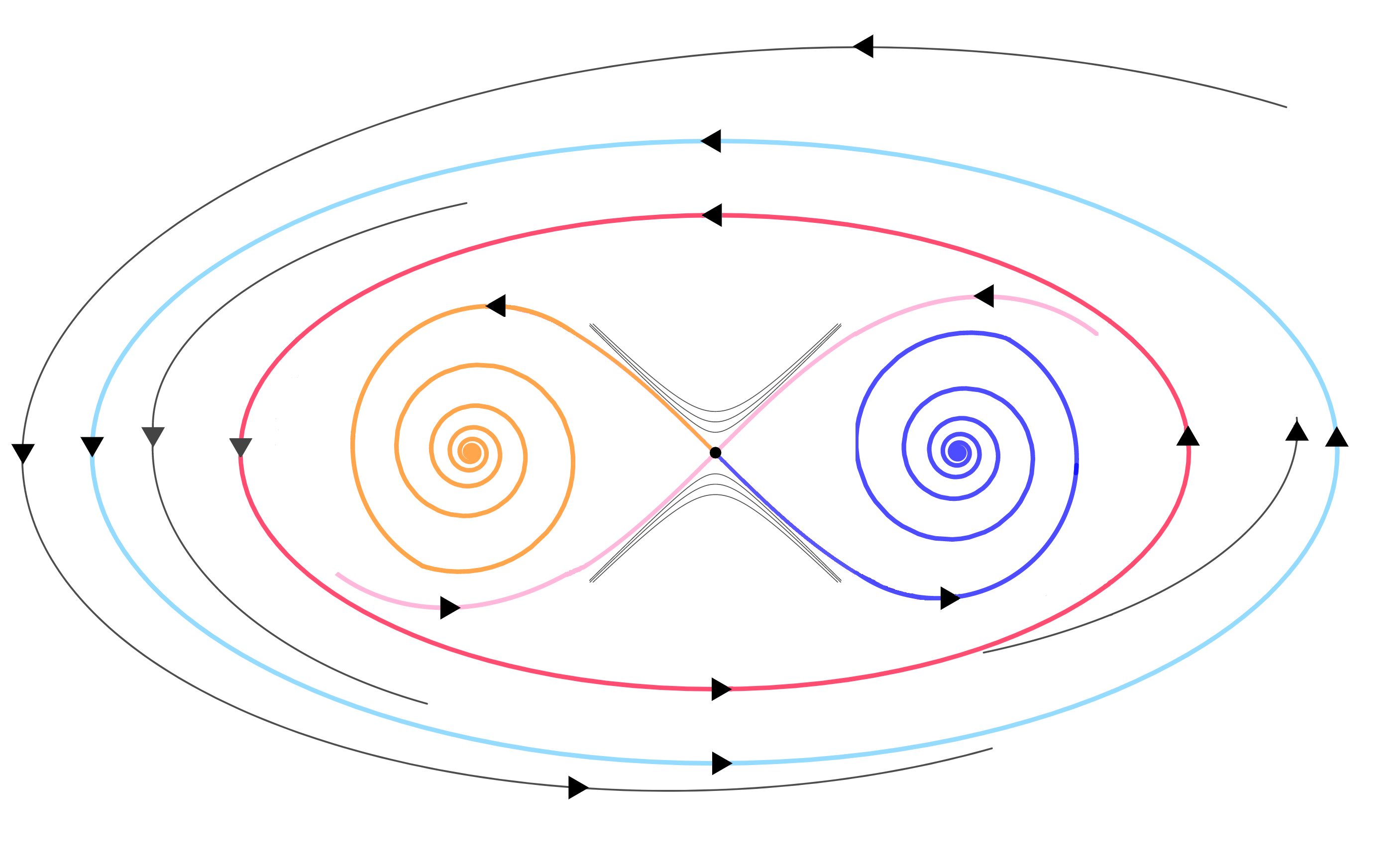}
	\caption{A hidden attracting periodic orbit (cyan) coexists with an unstable periodic orbit (red) born after a homoclinic bifurcation. The unstable separatrices (blue and orange) tend to the asymptotically stable equilibria. The stable separatrices (pink) tend to the single unstable periodic orbit in the negative direction of time.}
	\label{fig: SSPPunstable}
\end{figure}
\begin{figure}
	\centering
	\includegraphics[width=1\linewidth]{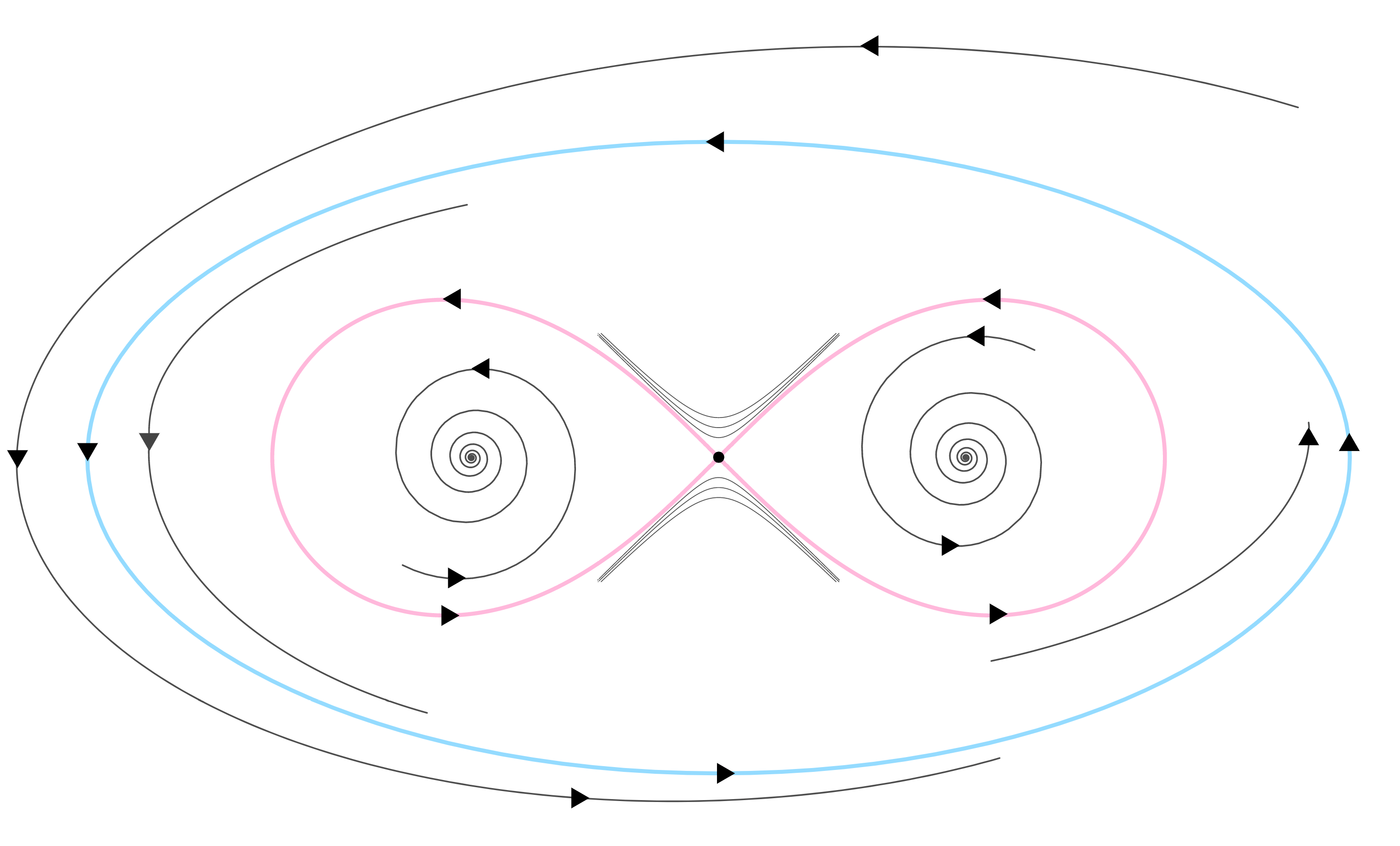}
	\caption{A self-excited attracting periodic orbit (cyan) coexists with a homoclinic ``figure eight'' (pink), which bifurcates into a single unstable periodic orbit (Fig. \ref{fig: SSPPunstable}) or a pair of unstable periodic orbits (Fig. \ref{fig: SSPPunstableTwo}).}
	\label{fig: SSPPHomo}
\end{figure}

It turns out that it is possible to suggest a complete qualitative description of the dynamics in the region of linear stability. Namely, using numerical estimates to bound the global attractor of \eqref{EQ: ElNinoSSmodel}, we provide analytical-numerical evidence for the existence of two-dimensional inertial manifolds with the aid of our developments on the theory \cite{Anikushin2020Geom,Anikushin2020FreqDelay,Anikushin2020Semigroups} (see Section \ref{SEC: IMSSjustification}). From this we can use planar arguments to conclude that the discovered self-excited oscillations must coexist with a pair of unstable periodic orbits (see Fig. \ref{fig: SSPPunstableTwo} and the region $\Omega_{se}$ in Fig. \ref{FIG: SSHiddenCurves}) and hidden oscillations must coexist with a single unstable periodic orbit (see Fig. \ref{fig: SSPPunstable} and the region $\Omega_{hid}$ in Fig. \ref{FIG: SSHiddenCurves}). Moreover, between these phase portraits there must exist a homoclinic ``figure eight'' (see Fig. \ref{fig: SSPPHomo}), which bifurcates into the single unstable orbit or the pair of unstable orbits. These parameters correspond to the upper hidden curve (orange) in Fig. \ref{FIG: SSHiddenCurves}. Moreover, on the lower hidden curve a saddle-node bifurcation of the hidden and unstable periodic orbits occurs, which leads to a gradient-like behavior in the region $\Omega_{grad}$ from Fig. \ref{FIG: SSHiddenCurves}. We conjecture that this scenario completely describes the qualitative dynamics in the region of linear stability (see Remark \ref{REM: SSLinearStabilityDescription}).

Note that the discovered phase portraits indicate the irregularity of ENSO (see R.~Kleeman \cite{Kleeman2008}, the monograph of M.J.~McPhaden, A.~Santoso and W.~Cai \cite{McPhadenBook2020} for discussions) more naturally than the dynamics corresponding to the region of linear instability (above the neutral curve), where all typical transient processes lead only to an attracting periodic orbit, for the following two reasons.

On the one hand, the observed multistability in the region $\Omega_{hid}$ of hidden oscillations described in Fig. \ref{fig: SSPPunstable} may correspond to the stochastic irregularity of ENSO \cite{Kleeman2008}. In this theory, the irregularity is believed to be caused by external forces (noise), which act on smaller time and space scales. Such noise may force the current system state to different basins of attraction. Note that Fig. \ref{fig: SSPPunstable} describes dynamics only schematically. In the Suarez-Schopf model the single unstable periodic orbit at some of its parts may be located close to zero (that agrees with the homoclinic bifurcation scenario) as well as to the hidden periodic orbit (especially near the warm and cold phases) or the unstable separatrices. This only promotes the role of noise (see Fig. \ref{FIG: SShidden1}).

On the other hand, it is well-known that a homoclinic ``figure eight'' (as in Fig. \ref{fig: SSPPHomo}) may lead to reach (chaotic) dynamics under the influence of a small periodic forcing (see, for example, the papers of S.V.~Gonchenko, C.~Sim\'{o} and A.~Vieiro \cite{Gonchenkoetal2013} or A.~Litvak-Hinenzon and V.~Rom-Kedar \cite{LitvakKedar1997}). This goes in the direction of the phenomena observed in a bit more complex delayed model studied by E.~Tziperman et al. \cite{Tzipermanetal1994}, who laid the foundation for the irregularity theory based on low-dimensional chaos caused by a small amplitude periodic forcing. In Section \ref{SEC: SSPeriodicForceChaos} we provide numerical evidence that a small periodic forcing can cause a similar chaotic behavior in the Suarez-Schopf model.

Thus, our investigation of the Suarez-Schopf model shows that the simple linear delayed interaction between the Rossby and Kelvin waves, which is described by the model, may serve as a basis for both theories of irregularity. Discovered patterns and relative simplicity of the model open a perspective of studying via more delicate analytical and numerical techniques. It also shows that not only engineering systems, but also models from climate dynamics must be studied more carefully when it comes to numerical experiments.

This paper is organized as follows. In Section \ref{SEC: IMSSjustification} we provide analytical-numerical evidence for the existence of two-dimensional inertial manifolds in the model. In Section \ref{SEC: SSmodelHidden} we state basic rigorous facts about the dynamics of the Suarez-Schopf model, show by means of concrete parameters from the region of linear stability that the presence of self-excited periodic orbits, hidden orbits and homoclinics is possible and propose a description of the dynamics in the region of linear stability. In Section \ref{SEC: SSPeriodicForceChaos} we demonstrate that a small periodic forcing can cause chaotic behavior in the model. In Section \ref{SEC: NonOscillatoryConj} we propose an analytically described nonoscillatory region motivated by dimension estimates. In Section \ref{SEC: Conclusions} we collect some conclusions of our investigations. In Appendix \ref{APP: LocalizationHiddenAttrators} we briefly discuss the theory of hidden attractors and their localization via the linear feedback gain method by means of the Suarez-Schopf model. In Appendix \ref{SEC: AsyncOscillLLTL} we use certain parameters from the Suarez-Schopf model to discover asynchronous oscillations, which can be hidden or self-excited, in a ring array of coupled lossless transmission lines.
\section{Analytical-numerical justification of the existence of two-dimensional inertial manifolds}
\label{SEC: IMSSjustification}

In this section, we provide analytical-numerical evidence for the existence of $C^{1}$-differentiable uniformly normally hyperbolic inertial manifolds in \eqref{EQ: ElNinoSSmodel}, which are given by a graph over the two-dimensional spectral subspace corresponding to the linearization at $\phi^{0}$. For purposes of Section \ref{SEC: SSPeriodicForceChaos} we will study a periodically forced model, although, our theory \cite{Anikushin2020Geom} is not limited to such models.

Let us firstly describe a general scheme from \cite{Anikushin2020FreqDelay,Anikushin2020Geom} by means of the following class of delay equations in $\mathbb{R}^{n}$ given by
\begin{equation}
	\label{EQ: ExampleDelayEqClass}
	\dot{x}(t) = \widetilde{A}x_{t} + \widetilde{B}F(t,Cx_{t})+W(t),
\end{equation}
where $x_{t}(\theta)=x(t+\theta)$ for $\theta \in [-\tau,0]$ is the history segment; $\widetilde{A} \colon C([-\tau;0];\mathbb{R}^{n}) \to \mathbb{R}^{n}$, $\widetilde{B} \colon \mathbb{R}^{m} \to \mathbb{R}^{n}$ and $C \colon C([-\tau,0];\mathbb{R}^{n}) \to \mathbb{R}^{r}$ are bounded linear operators; $W \colon \mathbb{R} \to \mathbb{R}^{n}$ is a $\sigma$-periodic continuous function and $F \colon \mathbb{R} \times \mathbb{R}^{r} \to \mathbb{R}^{m}$ is a $\sigma$-periodic in $t$ continuous map satisfying for some $\Lambda>0$ the Lipschitz inequality
\begin{equation}
	\label{EQ: SSExampleLipsschitzCond}
	|F(t,y_{1})-F(t,y_{2})|_{\Xi} \leq \Lambda |y_{1}-y_{2}|_{\mathbb{M}} \text{ for any } y_{1},y_{2} \in \mathbb{M}, t \in \mathbb{R}.
\end{equation}
Here $\Xi = \mathbb{R}^{m}$ and $\mathbb{M} = \mathbb{R}^{r}$ are endowed with some (not necessarily Euclidean) inner products and the corresponding norms are denoted by $|\cdot|_{\Xi}$ and $|\cdot|_{\mathbb{M}}$ respectively.

Put $\mathcal{Q} := \mathbb{R}/\sigma\mathbb{Z}$ and consider the shift dynamical system $\vartheta$ on $\mathcal{Q}$, i.~e. $\vartheta^{t}(q) := q + t$ for any $t \in \mathbb{R}$ and $q \in \mathcal{Q}$. For $t \geq 0$, $q \in \mathcal{Q}$ and $\phi_{0} \in \mathbb{E} := C([-\tau,0];\mathbb{R}^{n})$ we define the cocycle map as $\psi^{t}(q,\phi_{0}) := x_{t+t_{0}}$, where $t_{0} = q + k\sigma$ for some $k \in \mathbb{Z}$ and $x(t)=x(t;t_{0},\phi_{0})$ is the classical solution (see, for example, J.K.~Hale \cite{Hale1977}) to \eqref{EQ: ExampleDelayEqClass} defined for $t \geq t_{0}-\tau$ and such that $x_{t_{0}} = \phi_{0}$. Then it is clear that the cocycle property $\psi^{t+s}(q,\phi_{0}) = \psi^{t}( \vartheta^{s}(q), \psi^{s}(q,\phi_{0}))$ is satisfied for all $t,s \geq 0$, $q \in \mathcal{Q}$ and $\phi_{0} \in \mathbb{E}$.

By the Riesz representation theorem, there exist matrix-valued functions of bounded variation $a(\theta)$ and $c(\theta)$ such that
\begin{equation}
	\label{EQ: LinearOperatorsACRepresentation}
	\widetilde{A}\phi = \int_{-\tau}^{0} da(\theta)\phi(\theta) \text{ and } C\phi = \int_{-\tau}^{0} d c(\theta) \phi(\theta) \text{ for all } \phi \in C([-\tau,0];\mathbb{R}^{n}).
\end{equation}
Put $\alpha(p) :=  \int_{-\tau}^{0} e^{p \theta} d a(\theta)$, $\gamma(p) := \int_{-\tau}^{0} e^{p \theta} d c(\theta)$ and consider the \textit{transfer matrix} $W(p)=\gamma(p)(\alpha(p)-pI)^{-1}\widetilde{B}$. Note that $\operatorname{det}(\alpha(p) - p I)=0$ is the characteristic equation for $\dot{x}(t) = \widetilde{A}x_{t}$ (see \cite{Hale1977}) and its roots are the possible poles of $W(p)$. Since we are going to consider complex values of $p$, it is essential to note that $W(p)$ represents a linear operator between the complexifications $\Xi^{\mathbb{C}}$ and $\mathbb{M}^\mathbb{C}$ of the spaces $\Xi$ and $\mathbb{M}$ respectively.

Let $\mathcal{Q}(y,\xi)$ be a quadratic form of $y \in \mathbb{M}$ and $\xi \in \Xi$. By $\mathcal{Q}^{\mathbb{C}}$ we denote its Hermitian extension to $\mathbb{M}^{\mathbb{C}}$ and $\Xi^{\mathbb{C}}$, i.~e. $\mathcal{Q}^{\mathbb{C}}(y_{1}+iy_{2},\xi_{1} + i \xi_{2}) := \mathcal{Q}(y_{1},\xi_{1}) + \mathcal{Q}(y_{2},\xi_{2})$ for any $y_{1},y_{2} \in \mathbb{M}$ and $\xi_{1},\xi_{2} \in \Xi$. We suppose that $\mathcal{Q}(y,0) \geq 0$ for all $y \in \mathbb{M}$ and $\mathcal{Q}(C\phi_{1}-C\phi_{2},F(t,C\phi_{1}) - F(t,C\phi_{2})) \geq 0$ for any $t \in \mathbb{R}$ and $\phi_{1},\phi_{2} \in \mathbb{E}$.

We have the following theorem, which follows from our results in \cite{Anikushin2020Geom, Anikushin2020FreqDelay}.
\begin{theorem}
	\label{TH: ExampleIMdelay}
	Suppose that for some $\nu_{0}>0$ the characteristic equation $\operatorname{det}(\alpha(p) - p I) = 0$ has exactly $j$ roots with $\operatorname{Re}p > -\nu_{0}$ and has no roots with $\operatorname{Re}p = -\nu_{0}$. For $\mathcal{G}$ as above, let the frequency inequality
	\begin{equation}
		\label{EQ: FrequencyInequalityGeneral}
		\sup_{\xi \in \Xi^{\mathbb{C}}}\frac{\mathcal{G}^{\mathbb{C}}(-W(-\nu + i \omega)\xi,\xi)}{|\xi|^{2}_{\Xi^{\mathbb{C}}} } < 0 \text{ for all } \omega \in \mathbb{R}
	\end{equation}
    be satisfied. Suppose \eqref{EQ: ExampleDelayEqClass} has at least one bounded in the future solution. Then we have
    \begin{enumerate}
    	\item[1).] There exists a family of $j$-dimensional submanifolds $\mathfrak{A}(q)$, $q \in \mathcal{Q}$, which are mapped bi-Lipschitz homeomorphically by the spectral projector $\Pi$ onto the generalized eigenspace $\mathbb{E}^{u}(\nu_{0})$ corresponding to the roots with $\operatorname{Re}p > -\nu_{0}$. 
    	\item[2).] The family $\mathfrak{A}(q)$ is invariant, i.~e. $\psi^{t}(q,\mathfrak{A}(q)) =\mathfrak{A}(\vartheta^{t}(q))$, and $\psi^{t}(q,\cdot) \colon \mathfrak{A}(q) \to \mathfrak{A}(\vartheta^{t}(q))$ is a bi-Lipschitz homeomorphism for any $t \geq 0$ and $q \in \mathcal{Q}$.
    	\item[3).] If the derivative $F'_{y}$ of $F$ in $y$ exists and continuous, then $\mathfrak{A}(q)$ is a $C^{1}$-differentiable submanifold and the maps from items 1) and 2) are $C^{1}$-diffeomorphisms.
    	\item[4).] There exists a constant $M>0$ such that for any $q \in \mathcal{Q}$ and $\phi_{0} \in \mathbb{E}$ there exists a unique point $\phi^{*}_{0} \in \mathfrak{A}(q)$ such that
    	\begin{equation}
    		\| \psi^{t}(q,\phi_{0}) - \psi^{t}(q,\phi^{*}_{0}) \|_{\mathbb{E}} \leq M e^{-\nu_{0} t} \operatorname{dist}(\phi_{0},\mathfrak{A}(q)).
    	\end{equation}  
    \end{enumerate}
\end{theorem}
The family $\mathfrak{A}(q)$ is called an \textit{inertial manifold} for the cocycle $\psi$.

Note also that the projector $\Pi$ from Theorem \ref{TH: ExampleIMdelay} is stable w.~r.~t. perturbations. This is important for numerical simulations since the computation of such projectors requires to approximate integrals and eigenvalues as in formula \eqref{EQ: ProjectorFormulaSSmodel} below. Moreover, the inertial manifold itself is robust w.~r.~t. perturbations of the equation (see \cite{Anikushin2020Geom}).
\begin{remark}
	\label{REM: NormalHyperbolicity}
	 It can be shown that the inertial manifold is uniformly normally hyperbolic \cite{Anikushin2020Geom}. In the autonomous case, under the conditions of Theorem \ref{TH: ExampleIMdelay} this implies that the tangent space at any equilibrium is given by the generalized eigenspace corresponding to the characteristic roots with $\operatorname{Re}p > -\nu_{0}$. This circumstance is essential for the conclusions made at the end of Section \ref{SEC: SSmodelHidden}.
\end{remark}
\begin{remark}
	There is a natural choice of the quadratic form $\mathcal{Q}$ for general nonlinearities satisfying the Lipschitz condition \eqref{EQ: SSExampleLipsschitzCond}. Namely, $\mathcal{Q}(y,\xi) := \Lambda^{2}|y|^{2}_{\mathbb{M}}-|\xi|^{2}_{\Xi}$. Then the frequency inequality \eqref{EQ: FrequencyInequalityGeneral} reads as
	\begin{equation}
		\label{EQ: FrequencyInequalitySmith}
		|W(-\nu + i \omega)|_{\Xi^{\mathbb{C}} \to \mathbb{M}^{\mathbb{C}}} < \Lambda^{-1} \text{ for all } \omega \in \mathbb{R}.
	\end{equation}
	In the case of delay equations, \eqref{EQ: FrequencyInequalitySmith} with $\Xi$ and $\mathbb{M}$ being endowed with Euclidean inner products was used by R.A.~Smith \cite{Smith1992} to develop the Poincar\'{e}-Bendixson theory for such equations. An abstract operator form of \eqref{EQ: FrequencyInequalitySmith} (where $W(p)=C(A-pI)^{-1}B$ for certain operators $A,B,C$) can be used to study other classes of equations \cite{Anikushin2020Geom,Anikushin2021AAdyn}. For example, it can be applied to neutral delay equations \cite{Anikushin2020FreqDelay}. Moreover, in the case of semilinear parabolic problems with a self-adjoint linear part, it gives rise to the well-known Spectral Gap Condition in its optimal form (see \cite{Anikushin2020Geom,Anikushin2022FreqParab} for more discussions).
\end{remark}
\begin{remark}
	There is some flexibility in applications of Theorem \ref{TH: ExampleIMdelay} to concrete systems. Firstly, there may be many ways to write down a given system in the form \eqref{EQ: ExampleDelayEqClass} to obtain desired spectral properties in the linear part (see \cite{Smith1992} for a nice example). Moreover, if the nonlinearity $F$ is not Lipschitz (for example, as in \eqref{EQ: ElNinoSSmodel}), one applies a truncation procedure, redefining $F$ outside a dissipativity region $\Omega$, which contains the global attractor, such that the Lipschitz constant of $F$ on $\Omega$ (more precisely, on $C\Omega$) is preserved globally after the truncation. Since we wish to make the Lipschitz constant $\Lambda$ as small as possible, it is essential to provide sharp estimates for the region. As we will see, this is the main obstacle for rigorous proofs of the existence of inertial manifolds in the Suarez-Schopf model by our method\footnote{Here we mean proofs for the parameters from $\Omega_{se}$ and $\Omega_{hid}$. For certain parameters from $\Omega_{grad}$ the existence of two-dimensional and one-dimensional inertial manifolds can be proved rigorously \cite{Anikushin2020Semigroups}.}. Numerical experiments show that the estimate $\sqrt{1+\alpha}$ given in Section \ref{SEC: SSmodelHidden} is a rough bound for the global attractor. We expect that more delicate dissipativity regions can be constructed for \eqref{EQ: ElNinoSSmodel}.
\end{remark}

Let us consider a periodically perturbed Suarez-Schopf model
\begin{equation}
	\label{EQ: SSmodelForcedAbstract}
	\dot{x}(t) = x(t) - \alpha x(t-\tau) - x^{3}(t) + W(t).
\end{equation}
It can be shown that \eqref{EQ: SSmodelForcedAbstract} is dissipative (see \cite{Anikushin2020Semigroups}) and, consequently, its solutions are defined in the future and generate a cocycle for which there exists a global attractor $\mathcal{A}$ given by an invariant family of compact subsets $\mathcal{A}(q) \subset C([-\tau,0];\mathbb{R})$ for $q \in \mathcal{Q}$ (see \cite{Hale1977}).

For a fixed $R > 0$, let $\Lambda(R)>0$ be the Lipschitz constant of $y^{3}$ on $[-R,R]$ and let $g_{R}(y)$ be a $C^{1}$-differentiable function such that $g_{R}(y) = y^{3}$ for $y \in [-R,R]$ and $0 \leq g'_{R}(y) \leq \Lambda(R)$ for all $y \in \mathbb{R}$. We will apply Theorem \ref{TH: ExampleIMdelay} to the model
\begin{equation}
	\label{EQ: TruncatedSSGeneralPert}
	\dot{x}(t) = x(t) - \alpha x(t-\tau) - g_{R}(x(t)) + W(t).
\end{equation}
Clearly, \eqref{EQ: SSmodelForcedAbstract} coincides with \eqref{EQ: TruncatedSSGeneralPert} in the ball of radius $R$. 

Let us consider \eqref{EQ: TruncatedSSGeneralPert} in terms of \eqref{EQ: ExampleDelayEqClass} with $\widetilde{A}\phi = \phi(0) - \alpha \phi(-\tau)$, $\widetilde{B} = -1$, $C\phi = \phi(0)$ and $F(y) = g_{R}(y)$. For the quadratic form $\mathcal{Q}(y,\xi) = \xi(\Lambda y - \xi)$, the frequency inequality \eqref{EQ: FrequencyInequalityGeneral} reads as
\begin{equation}
	\label{EQ: SSTruncatedFreqCond}
	\operatorname{Re}W(-\nu_{0}+i\omega) + \Lambda^{-1}(r) > 0 \text{ for all } \omega \in \mathbb{R},
\end{equation}
where $W(p) = -( 1-\alpha e^{-\tau p} - p )^{-1}$. 

Note that the characteristic equation for the linear part of \eqref{EQ: TruncatedSSGeneralPert} is given by
\begin{equation}
	\label{EQ: SSZeroLinearization}
	1 - \alpha e^{-\tau p} - p = 0.
\end{equation}
It is not hard to prove that for $\alpha \in (0,1)$ and $\tau>0$ it has two leading real simple roots $\lambda_{1} > 0$ and $\lambda_{2} < 0$. We are interested in $\nu_{0}>0$ which separates $\lambda_{1}$ and $\lambda_{2}$ from the other eigenvalues in the sense of Theorem \ref{TH: ExampleIMdelay}. Numerical experiments suggest that for many interesting parameters the fibers $\mathcal{A}(q)$ lie in the ball of radius $1$. It can be verified, at least numerically, that for $R=1$, any $\alpha \in (0.5,1)$ and $\tau \in [1,1.98]$ the frequency inequality \eqref{EQ: SSTruncatedFreqCond} can be satisfied for a proper (in the given sense) $\nu_{0}$ (see Section 6 in \cite{Anikushin2020Semigroups}). Thus, for these parameters, if the fibers $\mathcal{A}(q)$ of the attractor $\mathcal{A}$ lie in the ball of radius $1$, then every $\mathcal{A}(q)$ is contained in the fiber $\mathfrak{A}(q)$ of a two-dimensional inertial manifold for \eqref{EQ: TruncatedSSGeneralPert}.

Now let us consider the spectral projector $\Pi$ onto the space $\mathbb{E}^{u}(\nu_{0})$ generated by the first two eigenfunctions $e^{\lambda_{1} \theta}$ and $e^{\lambda_{2}\theta}$. Standard calculations show that $\Pi$ is given by
\begin{equation}
	\label{EQ: ProjectorFormulaSSmodel}
	(\Pi\phi)(\theta) = c_{1} \cdot e^{\lambda_{1}\theta} + c_{2} \cdot e^{\lambda_{2}\theta}, \text{ where } c_{i} = \frac{ \phi(0) + \int_{-\tau}^{0}e^{-\lambda_{i}(\tau+\theta)} \phi(\theta)d\theta }{\alpha\tau e^{-\tau\lambda_{i}} - 1}, \ i \in{1,2}.
\end{equation}
It can be verified that since $\alpha \in (0,1)$, the denominator from the definition of $c_{i}$ is always non-zero, i.~e. $\lambda_{1}$ and $\lambda_{2}$ are simple. Below, we consider $\Pi$ as a map onto $\mathbb{R}^{2}$ given by $\Pi\phi := (c_{1},c_{2})$. Theorem \ref{TH: ExampleIMdelay} guarantees that its restriction to the inertial manifold fiber $\mathfrak{A}(q)$ is a $C^{1}$-diffeomorphism. Moreover, both the inertial manifold and the map $\Pi$ persist under small perturbations.
% !TeX spellcheck = en_US
\section{Oscillations in the region of linear stability}
\label{SEC: SSmodelHidden}
We start by stating some basic results concerned with dynamics of \eqref{EQ: ElNinoSSmodel}.

Let us consider the value $\gamma:=\sqrt{1+\alpha}$ and define the set $\mathcal{S}_{R}:=\{ \phi \in C([-\tau,0];\mathbb{R}) \ | \ \|\phi\|_{\infty} \leq \gamma + R \}$ for any $R \geq 0$. Here $\|\cdot\|_{\infty}$ denotes the supremum norm in the space of continuous functions $C([-\tau,0];\mathbb{R})$. At this point we can state some general dynamical properties of \eqref{EQ: ElNinoSSmodel} as follows.
\begin{enumerate}
	\item Equation \eqref{EQ: ElNinoSSmodel} generates a dissipative semiflow $\varphi^{t} \colon \mathbb{E} \to \mathbb{E}$, where $t \geq 0$, in the space $\mathbb{E}=C([-\tau,0];\mathbb{R})$. Moreover, the sets $\mathcal{S}_{R}$ are positively invariant, i.~e. $\varphi^{t}(\mathcal{S}_{R}) \subset \mathcal{S}_{R}$ for all $R \geq 0$, and the $\omega$-limit set of any point $\phi_{0} \in \mathbb{E}$ lies in $\mathcal{S}_{0}$.
	\item The $\omega$-limit set of any point $\phi_{0} \in \mathbb{E}$ satisfies the Poincar\'{e}-Bendixson trichotomy, i.~e. it can be either a stationary point, either a periodic orbit or a union of a set of stationary points and complete orbits connecting them.
\end{enumerate}
For Item 1 see, for example, our paper \cite{Anikushin2020Semigroups}. Item 2 follows from the result of J.~Mallet-Paret and G.R.~Sell \cite{MalletParetSell1996} since \eqref{EQ: ElNinoSSmodel} is a monotone cyclic feedback system. From Item 1 we have that there exists a global attractor, i.~e. an invariant set $\mathcal{A}$ which attracts bounded subsets of the phase space (see, for example, Propositions 1.6 and 1.7 in the monograph of N.V.~Kuznetsov and V.~Reitmann \cite{KuzReit2020}).

Now we turn to the local linear analysis. Recall that there are three stationary states: the origin $\phi^{0}(\cdot) \equiv 0$ and the pair of symmetric equilibria $\phi^{+}(\cdot) \equiv \sqrt{1-\alpha}$ and $\phi^{-}(\cdot) \equiv -\sqrt{1-\alpha}$. Eigenvalues of the linearization at $\phi^{0}$ are given by the roots $p \in \mathbb{C}$ of the equation  \eqref{EQ: SSZeroLinearization}. As we have already noted, \eqref{EQ: SSZeroLinearization} has exactly one positive root $\lambda_{1} > 0$, one negative root $\lambda_{2} < 0$ and the other roots are located to the left of the line $\lambda_{2} + i \mathbb{R}$. So there is always a one-dimensional unstable manifold at $\phi^{0}$.

Linearization at the symmetric equilibria leads to the characteristic equation
\begin{equation}
	\label{EQ: SSSymmetricLinearization}
	3\alpha - 2 - \alpha e^{-\tau p} - p = 0.
\end{equation}
Searching for purely imaginary roots $p = i \zeta$, we get that such roots appear for $\alpha>0.5$ and
\begin{equation}
	\zeta = \pm \sqrt{\alpha^{2} - (3\alpha - 2)^{2}} \text{ and } \tau = \frac{\pm\arccos\frac{3\alpha - 2}{\alpha} + 2\pi k}{\sqrt{\alpha^{2} - (3\alpha - 2)^{2}}},
\end{equation}
where $k=0,1,2,\ldots$. In particular, when $\alpha > 0.5$ is fixed, the first (as $\tau$ increases) pair of purely imaginary roots appear at
\begin{equation}
	\label{EQ: SSStabilityCurve}
	\tau = \frac{\arccos\frac{3\alpha - 2}{\alpha}}{\sqrt{\alpha^{2} - (3\alpha - 2)^{2}}}.
\end{equation}
Let $\Omega_{st}$ denote the set of pairs $(\tau,\alpha)$, where $\alpha \in (0,1)$ and $\tau>0$ such that the point $(\tau,\alpha)$ lies below the curve \eqref{EQ: SSStabilityCurve} which was named \textit{neutral curve} in \cite{Suarez1988}. This curve is the blue curve in Fig. \ref{FIG: SSHiddenCurves}. It can be verified that for any parameters from $\Omega_{st}$ equation \eqref{EQ: SSSymmetricLinearization} has only roots with negative real parts and, consequently, the symmetric stationary states $\phi^{+}$ and $\phi^{-}$ are asymptotically stable. We call $\Omega_{st}$ the \textit{region of linear stability}.
\begin{remark}
	For numerical integration of \eqref{EQ: ElNinoSSmodel}, we used the JiTCDDE package for Python (see G.~Ansmann \cite{AnsmannJITCODE2018}). For example, to obtain Fig. \ref{FIG: SelfExcitedSSmodel} and Fig. \ref{FIG: SShidden1} the system was integrated on the time interval $[0,1000]$ with integration parameters taken as $\operatorname{first\_step}=\operatorname{max\_step} = 10^{-3}$, $\operatorname{atol} = 10^{-5}$ and $\operatorname{rtol} = 10^{-5}$.
\end{remark}

Numerical simulations described in \cite{Suarez1988} (as well as in \cite{Boutleet2007ElNino}) did not detect periodic orbits for parameters from $\Omega_{st}$. In fact, a more careful study allows to detect self-excited periodic orbits for parameters $(\tau,\alpha) \in \Omega_{st}$ close to the neutral curve \eqref{EQ: SSStabilityCurve}. For example, when $\alpha = 0.75$, the point $(\tau,\alpha)$ belongs to $\Omega_{st}$ for any $\tau < 1.74$. Taking $\tau = 1.65$ and starting from a small neighborhood of $\phi^{0}$, one can observe a symmetric periodic orbit shown in Fig. \ref{FIG: SelfExcitedSSmodel}. Its period is $\sigma \approx 12.3$ ($\approx 8.17$ years, see Remark \ref{REM: SSmodelPeriodFormula}). Moreover, for $\alpha=0.75$ and $\tau = 1.58$ one may observe a hidden periodic orbit (Fig. \ref{FIG: SShidden1}) with period $\sigma \approx 15.3$ ($\sigma_{years} \approx 10.6$). It can be localized by a trajectory starting outside a small neighborhood of $\phi^{0}$. For example, the orange and blue curves in Fig. \ref{FIG: SShidden1} are obtained taking the initial data $\phi_{0}$ as the linear function with $\phi_{0}(0)=\pm 0.036$ and $\phi_{0}(-\tau)=\mp 0.036$ respectively.
\begin{remark}
	\label{REM: SSmodelPeriodFormula}
	To obtain the period in days, one should use the formula $\sigma_{days} = (\sigma \cdot \Delta) / \tau$, where $\Delta$ is the total delay (in days) depending on the Rossby and Kelvin waves propagation. This formula follows from the scaling used in \cite{Suarez1988} to obtain the dimensionless form \eqref{EQ: ElNinoSSmodel}. Usually, one takes $\Delta = 400$ \cite{Suarez1988} or $\Delta = 359$ \cite{Boutleet2007ElNino}. In the present section we use the value $\Delta=400$. Note that there exist hidden periodic orbits with a variety of periods such that any choice of $\Delta$ within reasonable limits does not change the qualitative conclusions on the behavior of periods.
\end{remark}
\begin{figure}
	\centering
	\includegraphics[width=1.\linewidth]{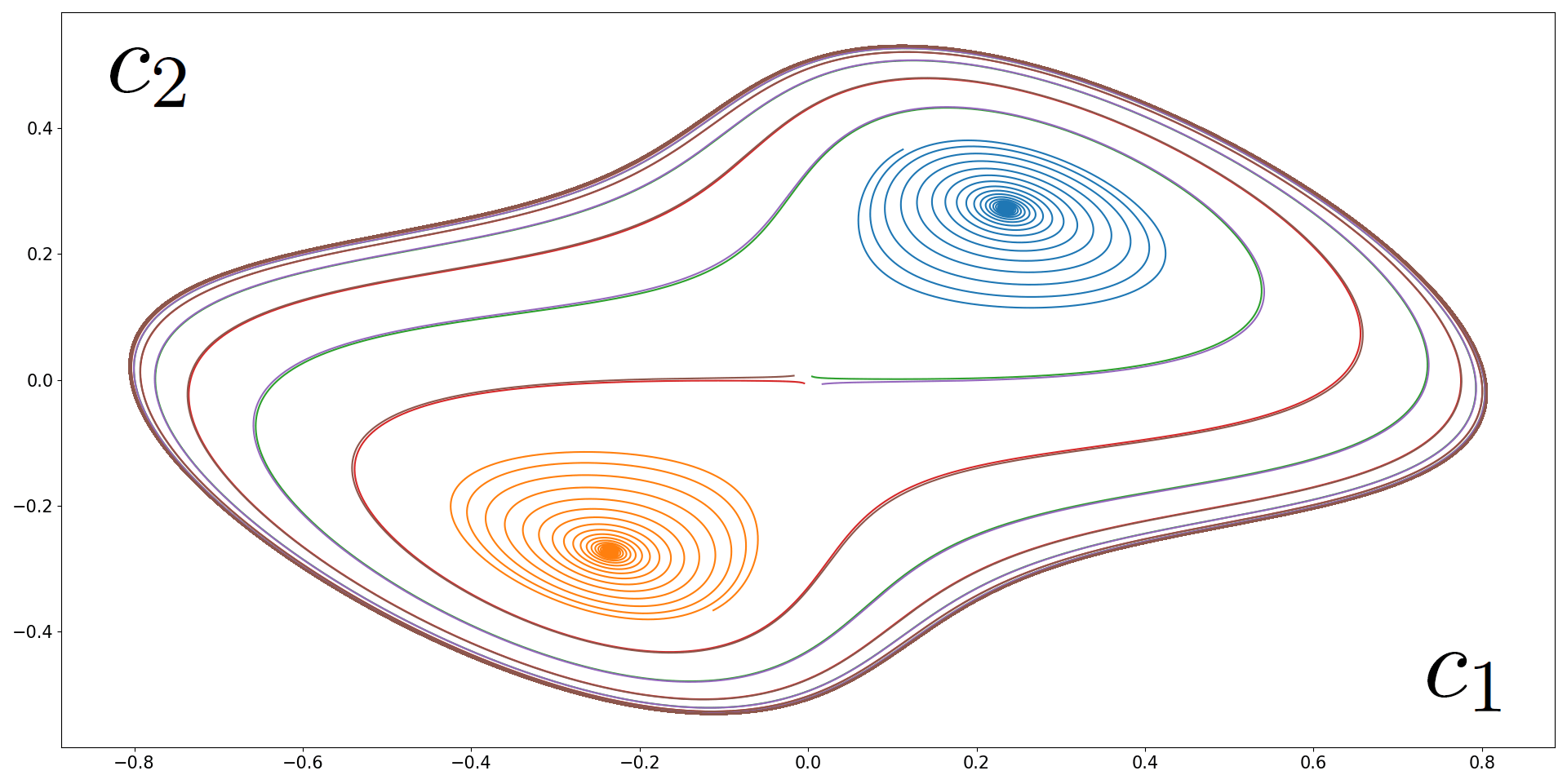}	
	\caption{A self-excited symmetric periodic orbit with period $\sigma \approx 12.3$ ($\sigma_{years} \approx 8.16$) of \eqref{EQ: ElNinoSSmodel}, where $\alpha=0.75$ and $\tau=1.65$, which can be localized from a small neighborhood of $\phi^{0}$. The blue and orange trajectories are attracted by the symmetric equilibria. All the trajectories are projected onto the $(c_{1},c_{2})$-plane by the projector $\Pi$.}
	\label{FIG: SelfExcitedSSmodel}
\end{figure}
\begin{figure}
	\centering
	\includegraphics[width=1.\linewidth]{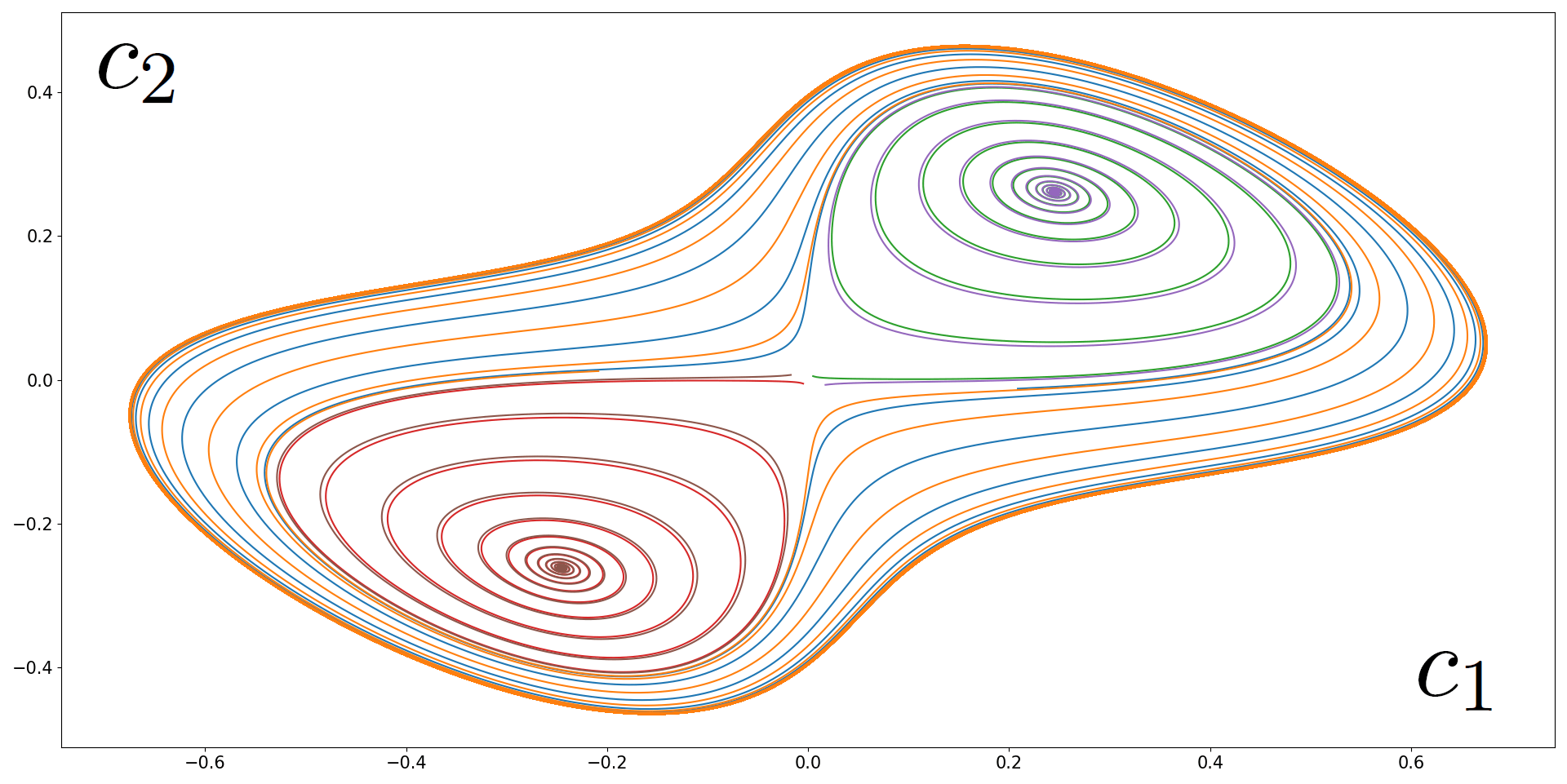}	
	\caption{A hidden symmetric periodic orbit with period $\sigma \approx 15.3$ ($\sigma_{years} \approx 10.6$) of system \eqref{EQ: ElNinoSSmodel}, where $\alpha=0.75$ and $\tau=1.58$, localized by the blue and orange trajectories. The other trajectories, starting sufficiently close to $\phi^{0}$, are attracted by the symmetric equilibria. All the trajectories are projected onto the $(c_{1},c_{2})$-plane by the projector $\Pi$.}
	\label{FIG: SShidden1}
\end{figure}

From numerical experiments it can be observed that for each pair of parameters in $\Omega_{st}$ there exists at most one attracting symmetric periodic orbit, which ``encloses'' the global attractor. Moreover, this orbit can be localized by a trajectory starting outside of the dissipativity region $\mathcal{S}_{0}$. This circumstance (along with the symmetricity) is justified by the existence of two-dimensional inertial manifolds in the model. Moreover, the parameters corresponding to self-excited periodic orbits are located closer to the neutral curve. They become hidden and then disappear as we decrease $\tau$ or $\alpha$.

Thus, to obtain the green and orange curves from Fig. \ref{FIG: SSHiddenCurves}, we estimated for any $\alpha \in [0.43,1)$ the minimal $\tau_{min}$ and the maximal $\tau_{max}$ values of the parameter $\tau$ for which hidden periodic orbits can be observed. Then the point $(\alpha,\tau_{min})$ corresponds to the lower hidden curve and $(\tau,\tau_{max})$ corresponds to the upper hidden curve. Some of approximations to these points are collected in Table \ref{TAB: TableHiddenPeriods}.
\begin{table}
	\begin{minipage}{.5\linewidth}
		%\caption{}
		\centering
		\begin{tabular}{||c c c c c||} 
			\hline
			$\alpha$ & $\tau_{min}$ & $\tau_{max}$ & $\sigma$ & $\sigma_{years}$ \\ [0.5ex] 
			\hline\hline
			%0.53 & 3.19 & 3.48 & 18.41 & 5.14 \\
			%\hline
			0.52 & 3.35 & 3.66 & 18.41 & 5.37 \\
			\hline
			0.51 & 3.53 & 3.92 & 18.41 & 5.14 \\
			\hline
			0.50 & 3.73 & 4.19 & 19.21 & 5.13 \\ 
			\hline
			0.49 & 3.95 & 4.50 & 19.81 & 4.8 \\
			\hline
			0.48 & 4.20 & 4.86 & 20.79 & 4.68 \\
			\hline
		\end{tabular}
	\end{minipage}%
	\begin{minipage}{.5\linewidth}
		\centering
		%\caption{}
		\begin{tabular}{||c c c c c||} 
			\hline
			$\alpha$ & $\tau_{min}$ & $\tau_{max}$ & $\sigma$ & $\sigma_{years}$ \\ [0.5ex] 
			\hline\hline
			%0.53 & 3.19 & 3.48 & 18.41 & 5.14 \\
			%\hline
			0.47 & 4.49 & 5.28 & 22.03 & 4.57 \\
			\hline
			0.46 & 4.82 & 5.80 & 22.74 & 4.3 \\
			\hline
			0.45 & 5.19 & 6.43 & 24.75 & 4.22 \\
			\hline
			0.44 & 5.65 & 7.24 & 26.70 & 4.04 \\
			\hline
			0.43 & 6.18 & 8.3 & 29.11 & 3.85 \\
			\hline
		\end{tabular}
	\end{minipage} 
	\caption{A table of points $(\alpha,\tau_{min})$ and $(\alpha,\tau_{max})$ approximating the lower and upper hidden curves respectively. For each point $(\alpha,\tau_{max})$ estimations of the dimensionless period $\sigma$ and the period in years $\sigma_{years}$ (see Remark \ref{REM: SSmodelPeriodFormula}) of the corresponding hidden periodic orbit are presented.}
	\label{TAB: TableHiddenPeriods}
\end{table}

Note that the presence of two-dimensional inertial manifolds (homeomorphic to the plane) motivates the existence of a homoclinic ``figure eight'' corresponding to the parameters $\alpha=0.75$ and some $\tau \in (1.58,1.65)$. It is expected that there is only one such value of $\tau$ and it corresponds to the parameters from the upper hidden curve. Clearly such a value cannot be founded numerically, but near it one may expect (due to computation errors) a behavior similar to that caused by a homoclinic orbit. Namely, trajectories with initial data near $\phi^{0}$ may tend to the symmetric equilibria as well as to the attracting periodic orbit. We found such a behavior near $\tau=1.596$. We will return to these parameters in Section \ref{SEC: SSPeriodicForceChaos}.

Moreover, using planar arguments in the spirit of the Poincar\'{e}-Bendixson theory, one can justify that the stable separatrices of $\phi^{0}$ on the inertial manifold\footnote{For their existence it is essential that the manifold is at least $C^{1}$-differentiable and normally hyperbolic, so the eigenvalues of the linearization at $\phi^{0}$ on the manifold coincide with $\lambda_{1}$ and $\lambda_{2}$ from \eqref{EQ: SSZeroLinearization}. See Remark \ref{REM: NormalHyperbolicity}.} in the case described by Fig. \ref{FIG: SelfExcitedSSmodel}, must tend in the negative direction of time to a pair of unstable periodic orbits enclosing $\phi^{+}$ and $\phi^{-}$ respectively. Thus, this case corresponds to the phase portrait described in Fig. \ref{fig: SSPPunstableTwo}. Analogously, in the case described by Fig. \ref{FIG: SShidden1}, the stable separatrices must tend in the negative direction of time to a single unstable periodic orbit, which separates $\phi^{0},\phi^{+}$ and $\phi^{-}$ from the hidden periodic orbit. This case corresponds to the phase portrait described in Fig. \ref{fig: SSPPunstable}.

\begin{remark}
	\label{REM: SSLinearStabilityDescription}
	We conjecture the following description of the dynamics in the region of linear stability. Fig. \ref{fig: SSPPunstable} describes a typical behavior in the region $\Omega_{hid}$ and Fig. \ref{fig: SSPPunstableTwo} describes a typical behavior in the region $\Omega_{se}$. Parameters from the upper hidden curve correspond to the existence of a homoclinic ``figure eight'' and the phase portraits described in Fig. \ref{fig: SSPPHomo}. This homoclinic bifurcates (note that $\lambda_{1} + \lambda_{2} > 0$, i. e. the saddle quantity is always positive; see Section \ref{SEC: NonOscillatoryConj}) into a single unstable periodic orbit (when moving towards the region $\Omega_{hid}$) or a pair of unstable periodic orbits (when moving towards the region $\Omega_{se}$). On the lower hidden curve a saddle-node bifurcation of the hidden and unstable periodic orbits occurs which lead to a gradient-like behavior in the region $\Omega_{grad}$. On the neutral curve a subcritical Andronov-Hopf bifurcation occurs.
\end{remark}
\section{Chaos caused by a periodic forcing}
\label{SEC: SSPeriodicForceChaos}
Let us consider \eqref{EQ: ElNinoSSmodel} under the influence of a periodic forcing as in \eqref{EQ: SSmodelForcedAbstract} with $W(t) = A\sin(t)$, where $A$ is a parameter.
\begin{remark}
\label{REM: IntegrationSSperturbedChaos}
For numerical integration of \eqref{EQ: SSmodelForcedAbstract}, here we use the JiTCDDE package for Python (see G.~Ansmann \cite{AnsmannJITCODE2018}). For the integration parameters we put $\operatorname{first\_step}=\operatorname{max\_step} = 10^{-4}$, $\operatorname{atol} = 10^{-6}$ and $\operatorname{rtol} = 10^{-6}$.
\end{remark}
\begin{figure}
	\begin{minipage}{.5\textwidth}
		\includegraphics[width=18pc,angle=0]{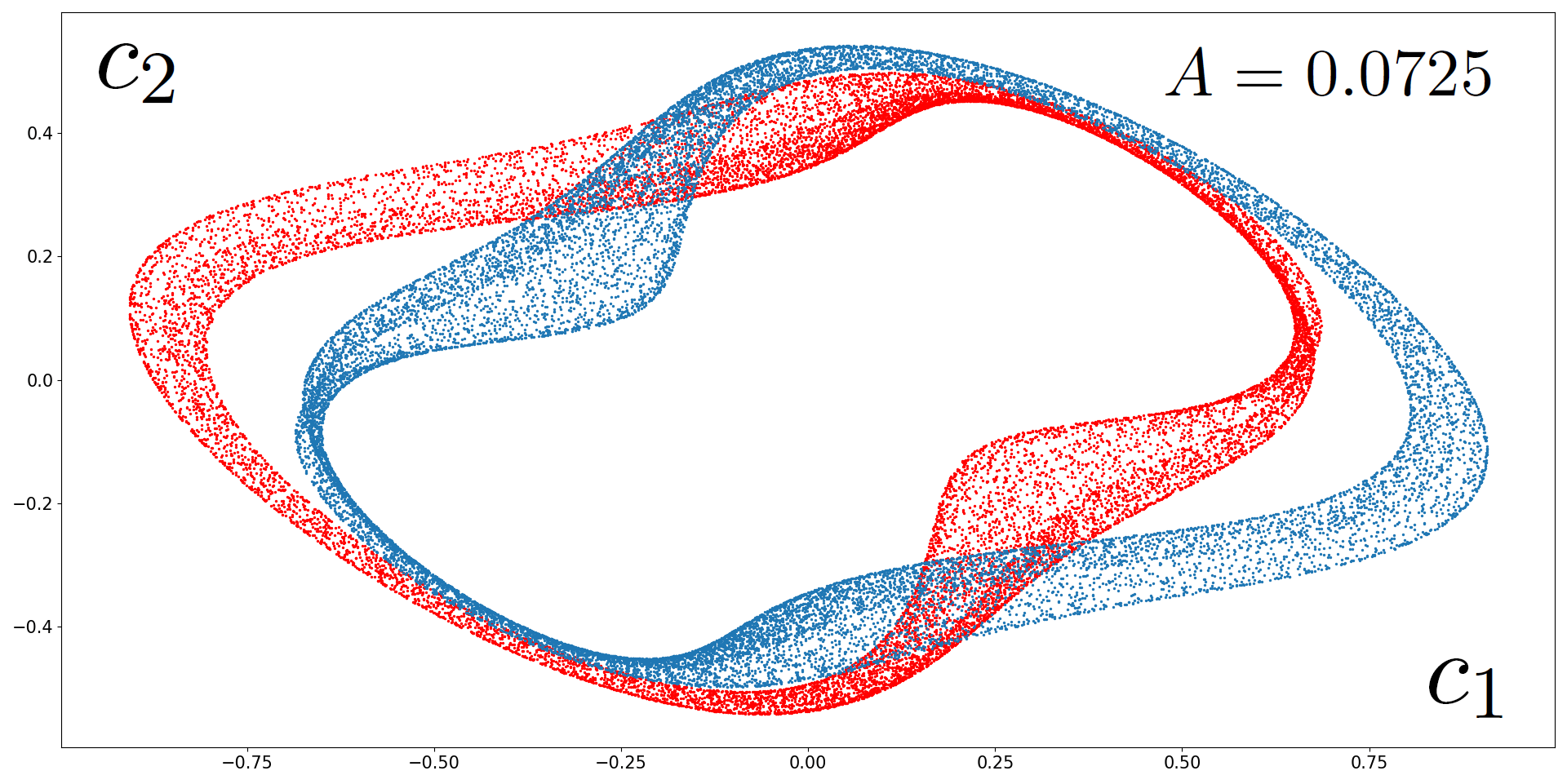}
		\includegraphics[width=18pc,angle=0]{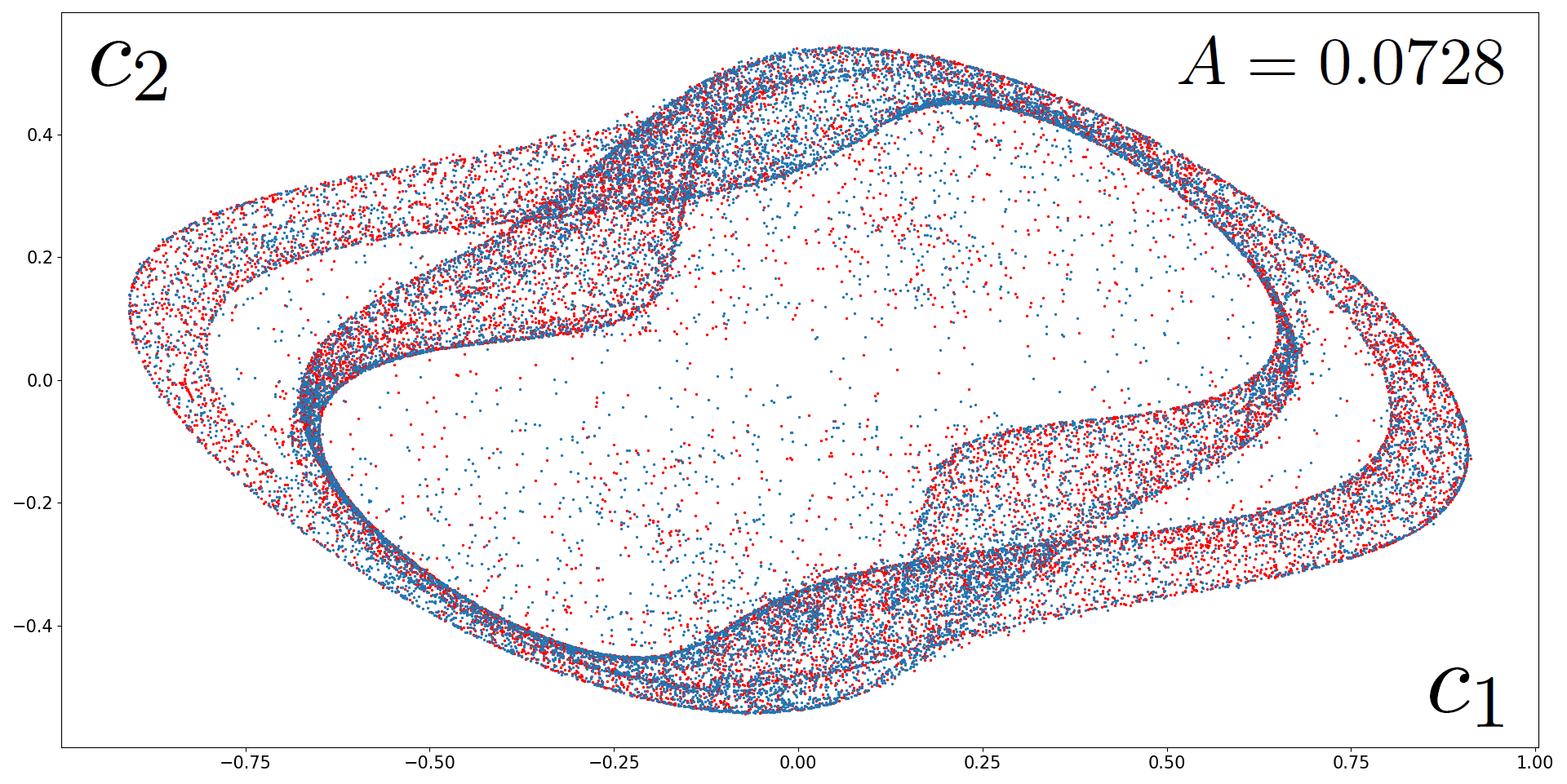}
		\includegraphics[width=18pc,angle=0]{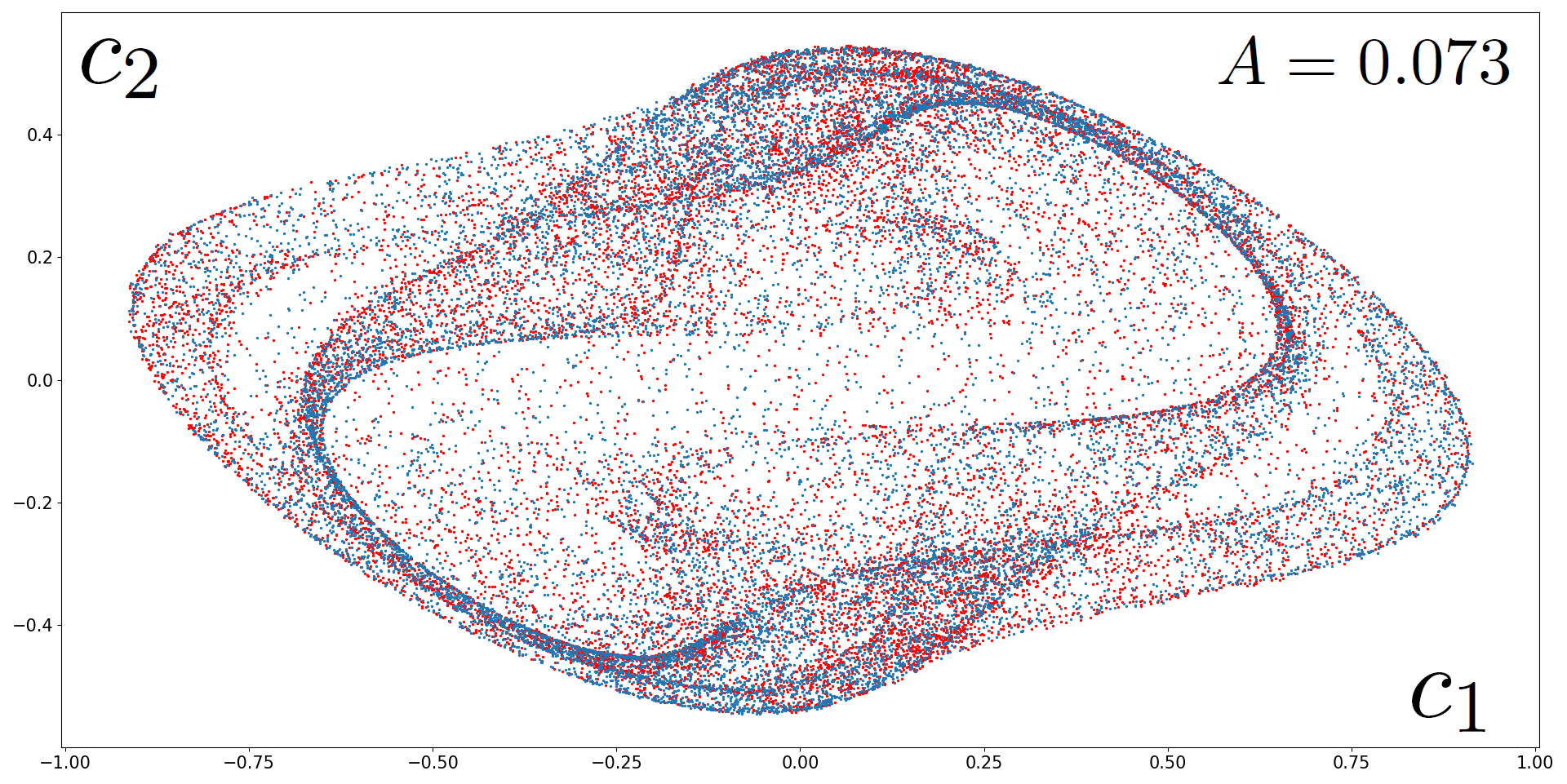}
	\end{minipage}%
	\begin{minipage}{.5\textwidth}
		\includegraphics[width=18pc,angle=0]{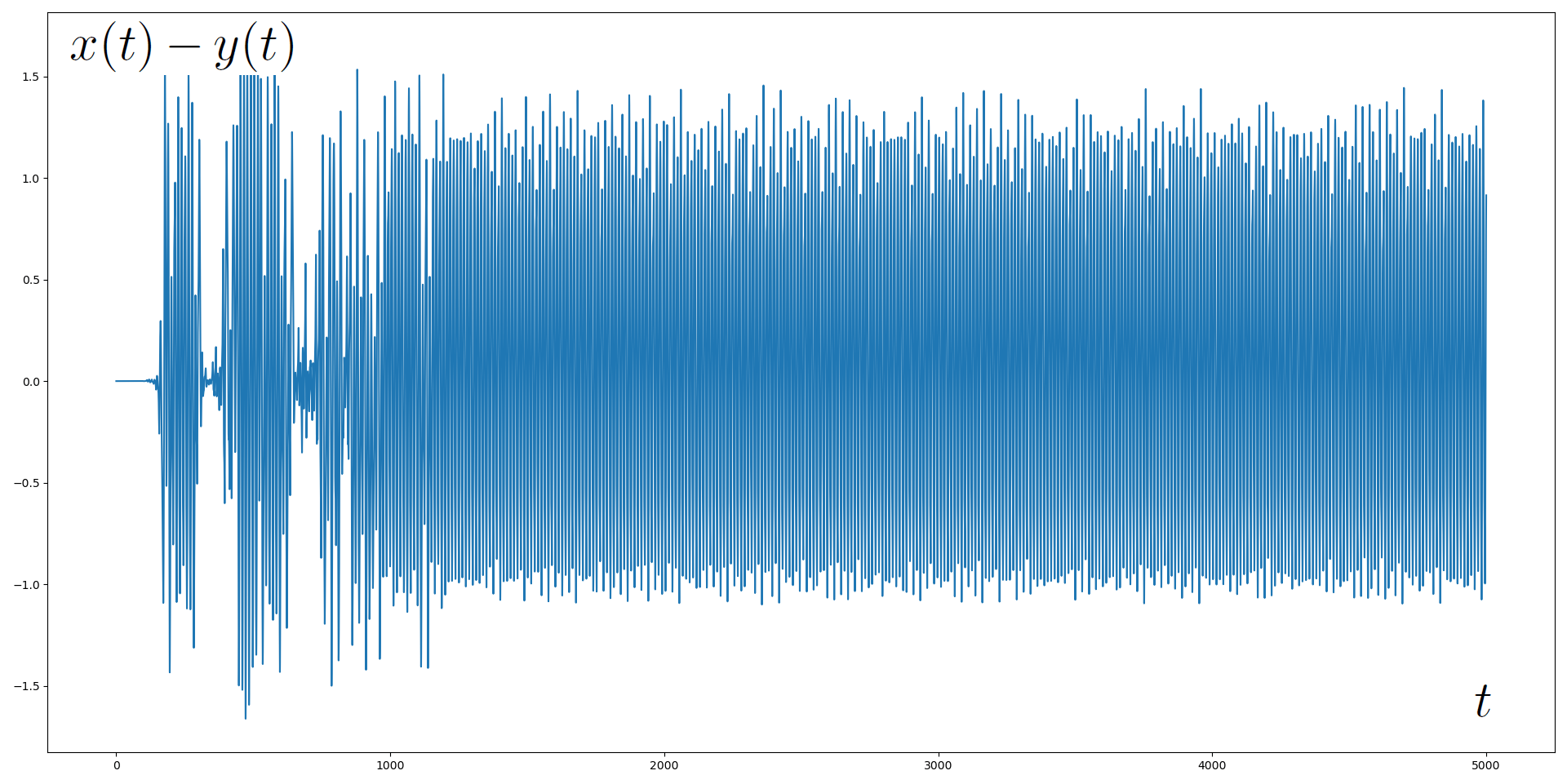}
		\includegraphics[width=18pc,angle=0]{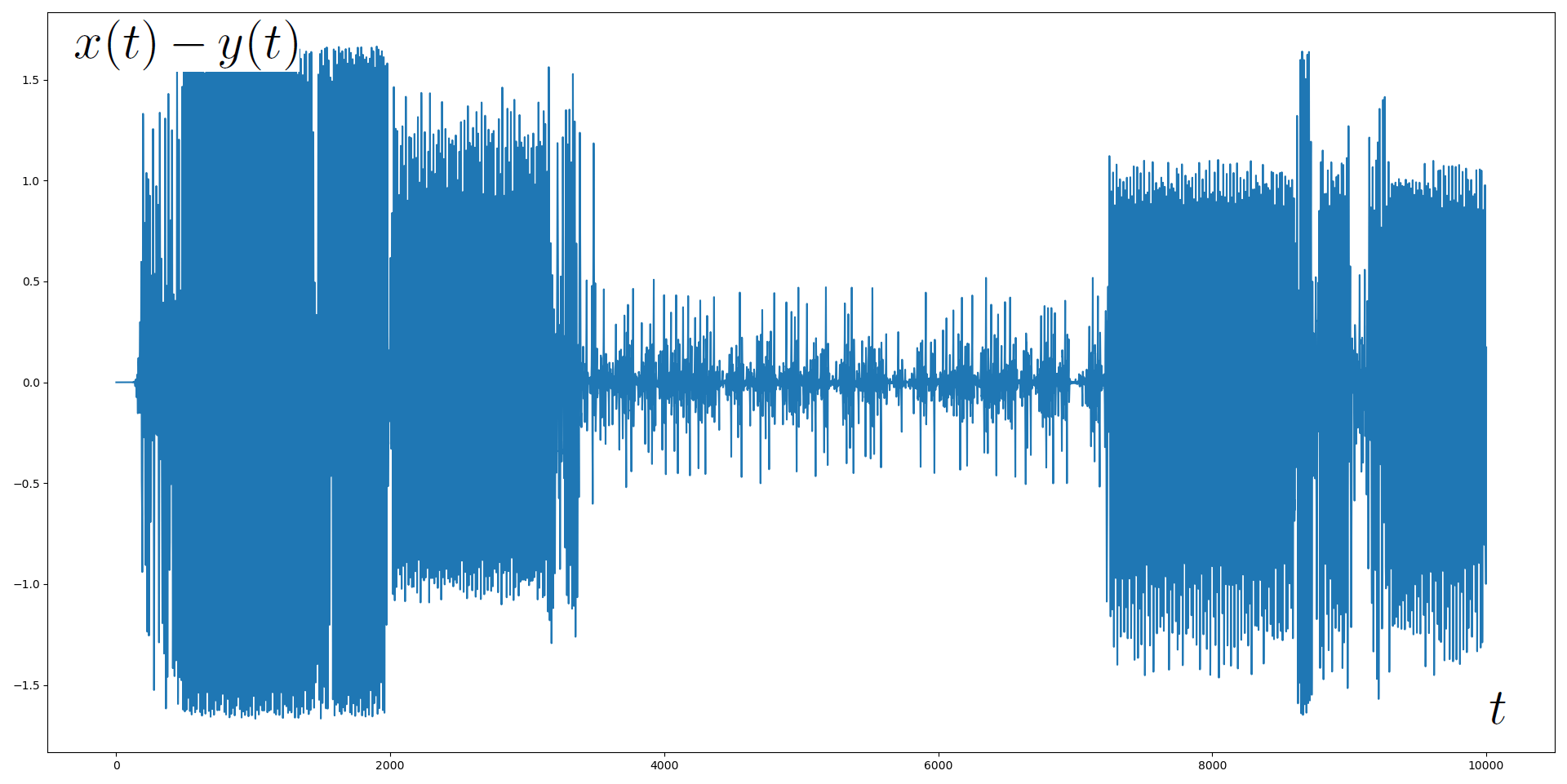}
		\includegraphics[width=18pc,angle=0]{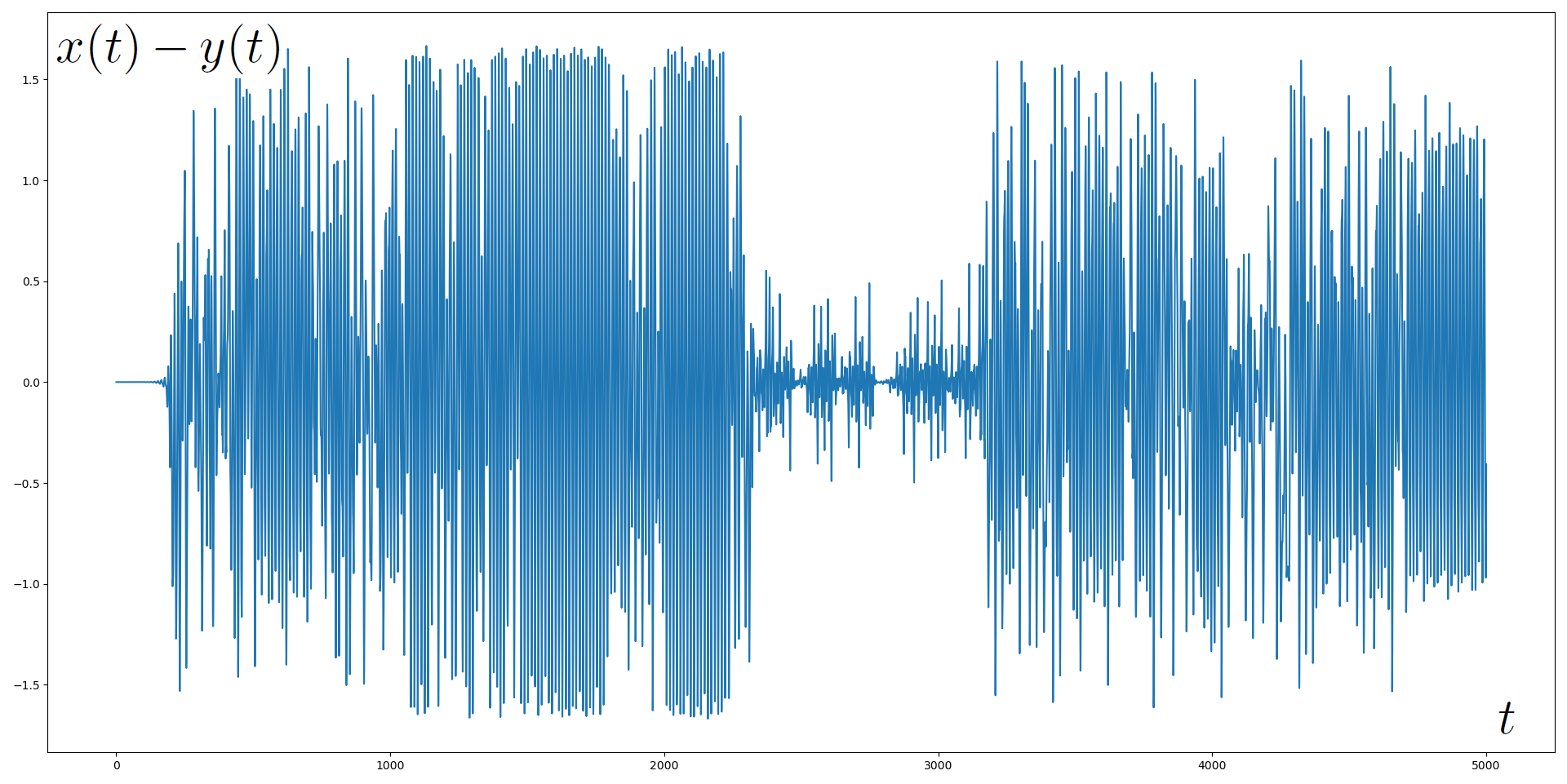}
	\end{minipage}%	
	\caption{(Left): Some of the limiting regimes for iterations of the Poincar\'{e} map for the solutions starting from $\phi_{1} \equiv 5$ (red) and $\phi_{2} \equiv 5 + 0.00001$ (blue) of \eqref{EQ: SSmodelForcedAbstract} with $W(t) = A\sin(t)$, $\tau=1.596$, $\alpha=0.75$ and different values of the amplitude $A$. All the iterations are projected onto the $(c_{1},c_{2})$-plane by the projector $\Pi$. (Right): The difference between the corresponding solutions.}
	\label{Fig: AmplitudeChaosandDifferenceCloseQC}
\end{figure}
We fix $\tau = 1.596$ and $\alpha=0.75$. It is likely that the model exhibits a similar as in the work E.~Tziperman et al. \cite{Tzipermanetal1994} route to chaos. In our case it seems to be connected with the presence of a homoclinic ``figure eight'' in a neighborhood of the parameter $(\alpha,\tau)=(0.75,1.596)$ as it was discussed in Section \ref{SEC: SSmodelHidden}. Thus the mechanism may be similar to the one described in S.V.~Gonchenko, C.~Sim\'{o} and A.~Vieiro \cite{Gonchenkoetal2013} or A.~Litvak-Hinenzon and V.~Rom-Kedar \cite{LitvakKedar1997}. It is, however, impossible to justify such routes numerically since they are linked with accurate computations of homoclinics and the presence of number-theoretic phenomena. Thus, we just note that an interesting phenomena can be observed for $A=0.0725$, $A=0.0728$ and $A=0.073$. Fig. \ref{Fig: AmplitudeChaosandDifferenceCloseQC} (left) shows projections onto $(c_{1},c_{2})$-plane as in \eqref{EQ: ProjectorFormulaSSmodel} for certain iterations of the Poincar\'{e} map applied to initial data $\phi_{1} \equiv 5$ (red) and $\phi_{2} \equiv 5+10^{-5}$ (blue) on the time interval $[600 \pi, 100000]$.

We also calculated the first two Lyapunov exponents along the trajectory of $\phi_{1}$ after the transient time $3000 \pi$ using $\operatorname{jitcdde\_lyap}$ procedure from the JiTCDDE package. For the integration we used the same parameters as in Remark \ref{REM: IntegrationSSperturbedChaos} and the linearized system was integrated on the time interval $[0,10000]$. Results are presented in Fig. \ref{FIG: LyapunovExponentsSSPert} for $A \in [0.068, 0.75)$ with the step equal to $0.0001$.
\begin{figure}
	\centering
	\includegraphics[width=1.\linewidth]{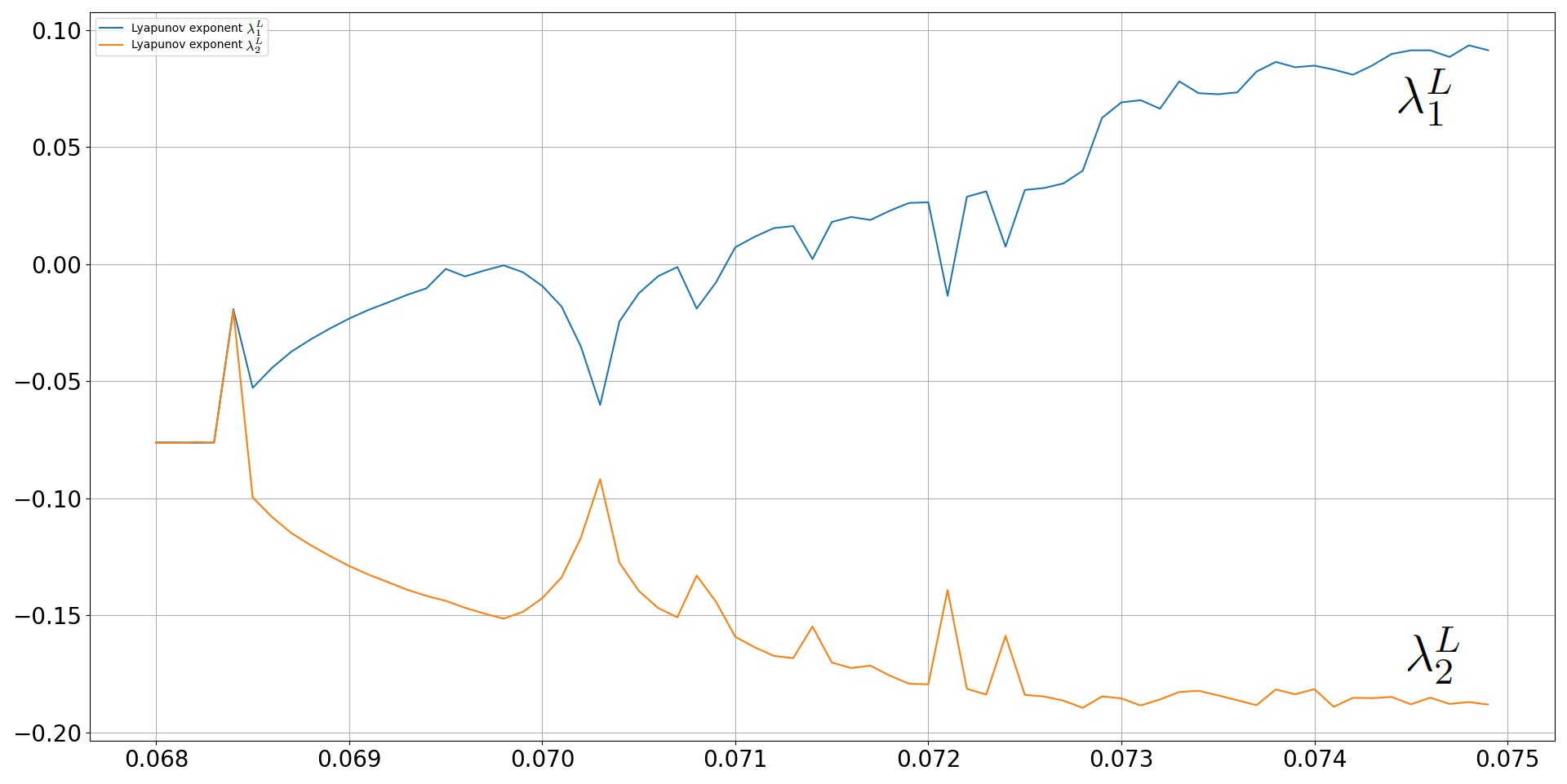}	
	\caption{Numerical approximations for the first two Lyapunov exponents $\lambda^{L}_{1}$ (blue) and $\lambda^{L}_{2}$ (orange) versus the amplitude $A$ (horizontal axis).}
	\label{FIG: LyapunovExponentsSSPert}
\end{figure}
For $A=0.073$ we also calculated the Lyapunov exponents over the trajectory of $\phi_{2}$. This gives the following two values: $\lambda^{L}_{1} = 0.0793 \pm 0.0008$ and $\lambda^{L}_{2} = -0.1834 \pm 0.0007$ and $\lambda^{L}_{1} = 0.057 \pm 0.0009$ and $\lambda^{L}_{2} = -0.1921 \pm 0.0007$ that justify a chaotic nature of the attractor.

From the perspective of climate dynamics, the forcing $W(t)$ should be a seasonal forcing and, consequently, its period should be $1$ year or, in the dimensionless form, $\sigma = 365\tau/\Delta \approx \tau$ (see Remark \ref{REM: SSmodelPeriodFormula}). However, for the considered parameters $(\alpha,\tau) = (0.75,1.596)$ we cannot observe any interesting behavior for forcings with such a period. Thus, the following problem is of interest.
\begin{problem}
	Can the perturbed model \eqref{EQ: SSmodelForcedAbstract} be chaotic for a small $\tau$-periodic forcing.
\end{problem}
Since the model also exhibits quasi-periodic behavior, it is interesting to study the occurrence of strange nonchaotic attractors (see \cite{Anikushin2021AAdyn} for a discussion).
\begin{problem}
	Can the perturbed model \eqref{EQ: SSmodelForcedAbstract} have strange nonchaotic attractors?
\end{problem}
\section{An analytical nonoscillatory region}
\label{SEC: NonOscillatoryConj}

Let $\lambda_{1}=\lambda(\alpha,\tau)>0$ and $\lambda_{2}=\lambda_{2}(\alpha,\tau)<0$ be the positive and negative roots of \eqref{EQ: SSZeroLinearization} respectively. From the dichotomy of linear autonomous systems (see Theorem 4.1, p. 181 in \cite{Hale1977}) it follows that the inequality $\lambda_{1}+\lambda_{2} < 0$ indicates the squeezing of two-dimensional volumes\footnote{To naturally speak about volumes, one should treat the equation in a proper Hilbert space setting. See \cite{Anikushin2020Semigroups} for details.} at the zero equilibrium $\phi^{0}$ of \eqref{EQ: ElNinoSSmodel}. In our work \cite{Anikushin2020Semigroups} we posed the following problem.
\begin{problem}
	\label{PROB: SSmodel1}
	Is it true that there are no periodic orbits and homoclinics in \eqref{EQ: ElNinoSSmodel} provided that $\lambda_{1}+\lambda_{2}<0$?
\end{problem}
The region in the space of parameters $(\tau,\alpha)$ determined by the inequality $\lambda_{1}+\lambda_{2}<0$, which we will denote as $\Omega_{dst}$, is displayed in Fig. \ref{FIG: DEnso2Dimensions}. Note that the boundary of $\Omega_{dst}$ is determined by $\lambda_{1}(\alpha,\tau) + \lambda_{2}(\alpha,\tau) = 0$. It is not hard to show that this is equivalent to
\begin{equation}
	\tau = \frac{\log \frac{1+\sqrt{1-\alpha^{2}}}{\alpha}}{\sqrt{1-\alpha^{2}}}.
\end{equation}
Clearly, $\Omega_{dst}$ is included into the region of linear stability $\Omega_{st}$. As we have shown, for the parameters from $\Omega_{st}$ the presence of self-excited and hidden periodic orbits is possible, but they are observed only above the lower hidden curve. Thus, our analytical-numerical investigations suggest that the answer to Problem \ref{PROB: SSmodel1} should be positive.

Note that a positive answer to Problem \ref{PROB: SSmodel1} implies the convergent (nonoscillatory) behavior in \eqref{EQ: ElNinoSSmodel} for the parameters from $\Omega_{dst}$ due to the Poincar\'{e}-Bendixson trichotomy discussed in Section \ref{SEC: SSmodelHidden}.

\begin{figure}
	\centering
	\includegraphics[width=1.\linewidth]{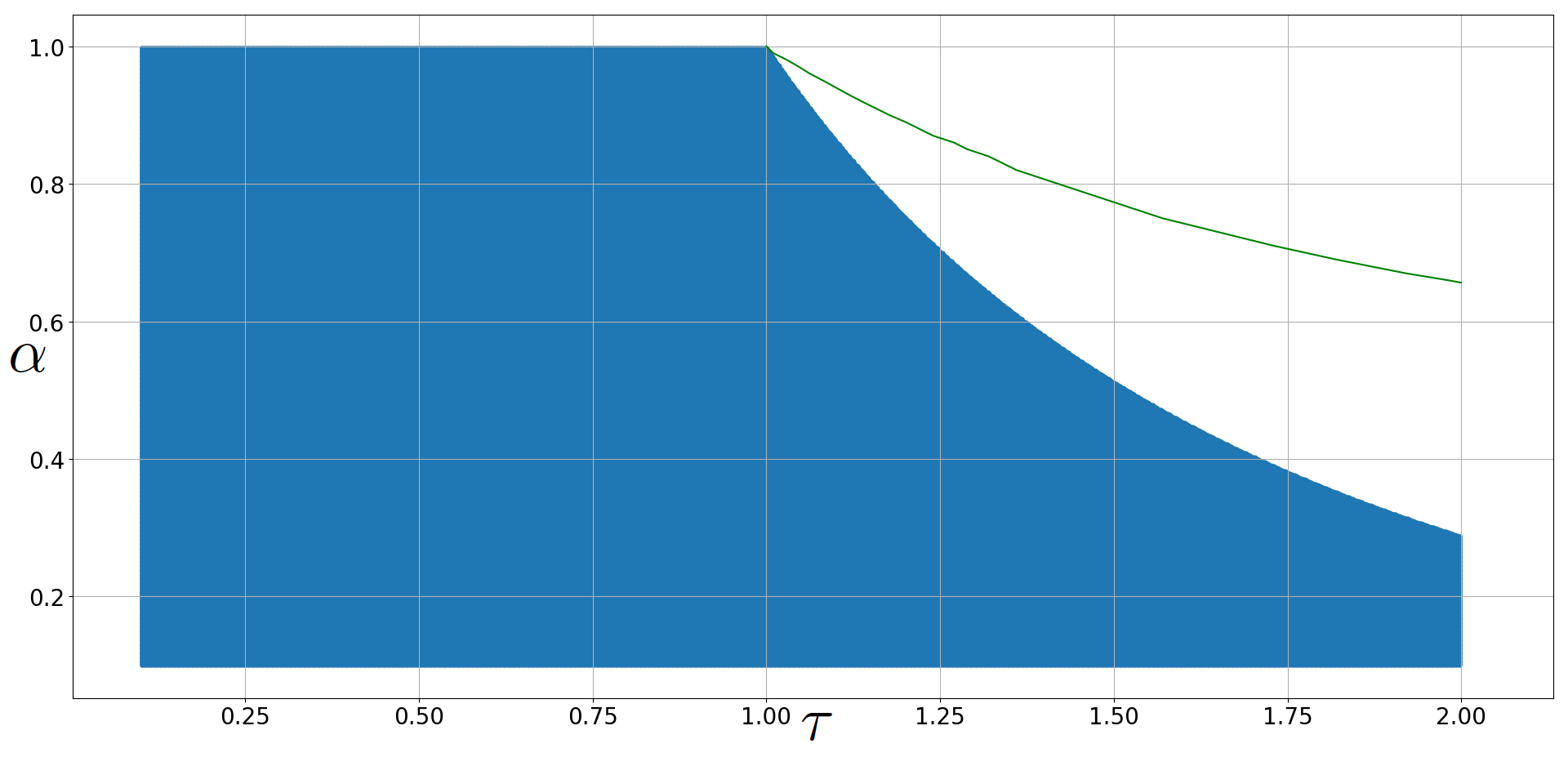}	
	\caption{A numerically obtained region (blue) in the space of parameters $(\tau,\alpha)$, where $0 \leq \tau \leq 2$, of system \eqref{EQ: ElNinoSSmodel}, for which there is a squeezing of two-dimensional volumes at the zero stationary state. The green curve is the lower hidden curve from Fig. \ref{FIG: SSHiddenCurves}.}
	\label{FIG: DEnso2Dimensions}
\end{figure}

Formulation of Problem \ref{PROB: SSmodel1} comes from dimension estimates \cite{Anikushin2020Semigroups}. A more stronger conjecture can be stated in terms of the Lyapunov dimension (see \cite{KuzReit2020}) as
\begin{problem}
	\label{PROB: SSmodel2}
	Is it true that the local Lyapunov dimension at the zero stationary state $\phi^{0}$ coincides with the Lyapunov dimension on the global attractor of \eqref{EQ: ElNinoSSmodel} for the parameters from $\Omega_{dst}$?
\end{problem}
This type of problems in dynamical systems is known as the Eden conjecture. For example, it is proved for the Lorenz system \cite{KuzReit2020}. Note that a positive answer to Problem \ref{PROB: SSmodel2} implies a positive answer to Problem \ref{PROB: SSmodel1} due to the criterion for the absence of invariant curves proved by M.Y.~Li and J.S.~Muldowney \cite{LiMuldowney1995}.

A close problem can be stated for the behavior of $k$-dimensional volumes. It is convenient to speak about it in terms of the so-called $k$-th compound cocycle that is a linear cocycle obtained as an extension to the $k$-th exterior power of the linearization cocycle over the global attractor. It is known that its growth exponent (the right endpoint of the Sacker-Sell spectrum) can be described through a Lyapunov exponent over some ergodic measure. In the case of the Suarez-Schopf model, this restricts us to the study of spectra over equilibria or periodic orbits. In \cite{MalletParetNussbaum2013} J.~Mallet-Paret and R.D.~Nussbaum obtained a nontrivial monotonicity result for scalar linear inhomogeneous equations that allows to compare the monodromy operators of $k$-th compound cocycles over distinct periodic orbits and even (in the case of the Suarez-Schopf model) values of $k$. This can be used to partially answer Problem \ref{PROB: SSmodel1} rigorously. We refer to our work \cite{Anikushin2020Semigroups} for examples and discussions in this direction.
\section{Conclusion}
\label{SEC: Conclusions}
Our investigation of the Suarez-Schopf model shows that the simple linear delayed interaction between the Rossby and Kelvin waves, which is described by the model, may serve as a basis for both theories of irregularity. It also shows that not only engineering systems, but also models from climate dynamics must be studied more carefully when it comes to numerical experiments.

The discovered patterns and relative simplicity of the model open a perspective of its studying via more delicate analytical and numerical techniques. These may include the construction of sharper regions of dissipativity and further development of the spectral theory for linear scalar equations with applications for the construction of inertial manifolds and dimension estimates.

We justified that the influence of a small periodic forcing on the model can cause irregular behavior similar to the chaos studied in the well-known periodically forced oscillators on the plane.

The developed intuition can be extended to study other models with delay such as the oscillatory networks from Appendix \eqref{SEC: AsyncOscillLLTL}, where certain parameters from the Suarez-Schopf model is used to discover hidden and self-excited asynchronous oscillations.
\appendix
\section{Hidden attractors and their localization via the linear feedback gain}
\label{APP: LocalizationHiddenAttrators}
In dynamical systems, a \textit{hidden attractor} is an attracting set, which cannot
be localized by a trajectory starting from a small neighborhood of any
equilibrium. This concept is motivated by the standard approach for numerical studying of nonlinear systems concerned with the local analysis of equilibria. Here trajectories from neighborhoods of unstable equilibria are traced to localize \textit{self-excited attractors} in the system. The notion of a hidden attractor was suggested by G.A.~Leonov and N.V.~Kuznetsov \cite{LeoKuz2013}, who also justified its significance for applied and theoretical problems. Since then, this area has attracted more and more attention. Hidden attractors were discovered in many applied models (besides \cite{LeoKuz2013}, see the reviews of D.~Dudkowski et al. \cite{Dudkowski2016} and N.V.~Kuznetsov \cite{Kuznetsov2020HOReview}), where their presence may lead to a sudden switch to unpredictable behavior and disastrous consequences. This indicates that one should be very careful when analyzing nonlinear systems.

When parameters leading to a hidden oscillation are found, it can be localized by different methods, including the guessing of an initial point from its basin of attraction as a result of a more careful treatment. Sometimes, among the system natural parameters there may exist parameters corresponding to self-excited attractors, which can evolve into the hidden attractor after certain variations of parameters (this is also the case we encounter in Section \ref{SEC: SSmodelHidden}). But tracking the evolution of all possible self-excited attractors in the space of parameters is a practically impossible task and it can be accomplished only in the simplest cases. Thus, the theory of hidden attractors is focused on the development of intuition in the localization of subregions in the space of parameters, where the presence of hidden attractors can be expected and to which more attention should be paid.

Many results on the localization of hidden oscillations are based on the continuation by parameter procedure (another approach is concerned with the so-called perpetual points \cite{Dudkowski2016}). Here we consider a family of parameterized, say by $\varepsilon \in [0,1]$, vector fields such that the corresponding to $\varepsilon=0$ system has an easily localizable (for example, self-excited) periodic orbit. Then we track the evolution of this orbit under changes of $\varepsilon$ with the hope that a nontrivial attractor of the system at $\varepsilon=1$, which corresponds to the original system, will be revealed. 

The most nontrivial part is concerned with the choice of the family of vectors fields. To the best of our knowledge, there is essentially one approach that has repeatedly proven its effectiveness for applied models often leading to nontrivial discoveries. We call it the \textit{linear feedback gain method}. Here, a linear feedback with gain parameters is applied to the system linearized at a given asymptotically stable equilibrium. Outside a neighborhood of the equilibrium, this feedback may be saturated to ensure, for example, that the equilibrium is unique and the system is dissipative. The gain parameters are taken to guarantee the required spectral properties, which make it possible to analytically justify the existence of an easily localizable periodic orbit under some additional assumptions.

The first success (namely, the discovery of a hidden chaotic attractor in the Chua circuit) in the field is concerned with the linear feedback gain method and its justification via the describing function method proposed by G.A.~Leonov and N.V.~Kuznetsov \cite{LeoKuz2013}. To the best of our knowledge, implementations of this method for infinite-dimensional systems (in the context of the theory of hidden oscillations) are still awaiting developments.

For us, it is more preferable the method suggested by I.M.~Burkin \cite{Burkin2014Hidden} in the case of ODEs. It is based on the generalized Poincar\'{e}-Bendixson theory developed by R.A.~Smith \cite{Smith1992}. To apply this theory, the gain parameters are chosen so that the equilibrium become unstable with a two-dimensional unstable manifold. Then the feedback is saturated preserving uniqueness of the equilibrium and providing dissipativity for the system (already at this point one can obtain an interesting region of parameters). The final condition is a frequency-domain condition, which in this case has the form of an inequality containing the transfer function and the gain parameters. If this condition is satisfied, then there is a two-dimensional inertial (slow) manifold and the Poincar\'{e}-Bendixson trichotomy holds. A bit delicate study of the inertial manifold properties allows to show the existence of a periodic orbit, which attracts typical points from a neighborhood of the equilibrium. 

In our work \cite{Anikushin2020Geom} (see also \cite{Anikushin2020Red,Anikushin2020Semigroups}), it is shown that these ideas have a natural geometric generalization that extends the area of applications. In particular, this geometric theory can be applied for delay equations due to the recent progress on the Frequency Theorem \cite{Anikushin2020FreqDelay} and semigroups in Hilbert spaces \cite{Anikushin2020Semigroups} done by one of the present authors. Of course, the experimentalist must be convinced by examples, some of which in the case of ODEs are already given by I.M.~Burkin and N.N.~Khien \cite{Burkin2014Hidden}.

Let us show applications of the linear feedback gain method by means of the Suarez-Schopf model. For this we consider \eqref{EQ: ElNinoSSmodel} in the form
\begin{equation}
	\label{EQ: SSmodelRewrited}
	\dot{x}(t) = (3\alpha - 2)x(t)  - \alpha x(t-\tau) + f(x(t)),
\end{equation}
where $f(y) = -y^{3} + 3( 1-\alpha ) y$. In terms of \eqref{EQ: ExampleDelayEqClass} we have $n=m=r=1$, $\widetilde{A}\phi = (3\alpha - 2) \phi(0) - \alpha \phi(-\tau)$, $\widetilde{B} = 1$, $C\phi = \phi(0)$. Thus, the transfer function of \eqref{EQ: SSmodelRewrited} is given by
\begin{equation}
	\label{EQ: SSmodelLinearStableEqTransfer}
	W(p) = \frac{1}{3\alpha - 2 - \alpha e^{-\tau p} - p}.
\end{equation}
For some numbers $0 < \mu_{\infty} < \mu$ (to be determined) we consider the nonlinearity (feedback gain) $g \colon \mathbb{R} \to \mathbb{R}$ defined as
\begin{equation}
	g(y) = \begin{cases}
		\mu_{\infty} \cdot (y-1) + \mu \text{ for } y > 1,\\
		\mu \cdot y \text{ for } |y| \leq 1, \\
		\mu_{\infty} \cdot (y+1) - \mu \text{ for } y < - 1.
	\end{cases}
\end{equation}
Along with \eqref{EQ: SSmodelRewrited} we consider the family of equations depending on $\varepsilon \in [0,1]$ as
\begin{equation}
	\label{EQ: ElNinoSSmodelFamily}
	\dot{x}(t) = (3\alpha - 2)x(t)  - \alpha x(t-\tau) + F_{\varepsilon}(x(t)),
\end{equation}
where $F_{\varepsilon}(y)=\varepsilon f(y) + (1-\varepsilon) g(y)$. Linearization of \eqref{EQ: ElNinoSSmodelFamily} with $\varepsilon=0$ at the zero leads to the characteristic equation
\begin{equation}
	\label{EQ: SSLinearizedGained}
	3\alpha - 2 + \mu - \alpha e^{-\tau p} - p = 0.
\end{equation}
Assuming that $-4\alpha + 2 < \mu < -2\alpha + 2$, as in Section \ref{SEC: SSmodelHidden} we get the curve (for a fixed $\alpha$)
\begin{equation}
	\label{EQ: SSmodelNetralCurveMu}
	\tau = \frac{\arccos \frac{3\alpha - 2 + \mu}{\alpha} }{\alpha^{2}-(3\alpha - 2 + \mu)^{2}},
\end{equation}
whose points $(\mu,\tau)$ corresponds to the first (as $\tau$ increases) appearance of purely imaginary roots of \eqref{EQ: SSLinearizedGained}. In particular, for $\tau = 1.58$ and $\alpha = 0.75$, putting $\mu=0.45$ (the point $(\mu,\tau)$ is above the curve) and $\mu_{\infty} = 0.005$ (the point $(\mu_{\infty},\tau)$ is below the curve), we get that \eqref{EQ: ElNinoSSmodelFamily} with $\varepsilon=0$ has a unique stationary state with a two-dimensional unstable manifold and the system is dissipative (for the latter see \cite{Smith1992}).

Recall that for $\tau = 1.58$ and $\alpha = 0.75$ all the roots of \eqref{EQ: SSSymmetricLinearization} are given by pairs of complex-conjugate numbers with negative real parts. Moreover, numerical calculations show that the first (as the real part decreases) two pairs of roots can be estimated as $\lambda_{1,2} \approx -0.05 \pm i 0.75$ and $\lambda_{3,4} = -1.2 \pm i 4.78$. Thus, for $\nu_{0} := 0.88$ there are exactly two roots located to the right of the line $-\nu_{0} + i \mathbb{R}$. It is clear that the nonlinearity $F_{0}=g$ is Lipschitz with the Lipschitz constant $\Lambda = \mu = 0.45$. It can be verified numerically that for the transfer function from \eqref{EQ: SSmodelLinearStableEqTransfer} we have
\begin{equation}
	\label{EQ: FreqSSmodelLocalization}
	|W(-\nu_{0} + i \omega)| \leq 0.64 < \Lambda^{-1} = 2.\overline{2} \text{ for all } \omega \in \mathbb{R}.
\end{equation}

Now we can state an auxiliary proposition as follows.
\begin{proposition}
	\label{PROP: SSmodelHidden}
	Consider \eqref{EQ: ElNinoSSmodelFamily} with $\varepsilon=0$, $\alpha = 0.75$, $\tau=1.58$, $\mu = 0.45$ and $\mu_{\infty}=0.005$ and let \eqref{EQ: FreqSSmodelLocalization} be satisfied. Then there exists a periodic orbit such that any point from a sufficiently small neighborhood of the zero equilibrium (except the points from its stable manifold) tends to the periodic orbit.
\end{proposition}
\begin{proof}
	Since the frequency inequality \eqref{EQ: FreqSSmodelLocalization} holds, Theorem \ref{TH: ExampleIMdelay} gives the existence of a two-dimensional inertial manifold $\mathfrak{A}$ (homeomorphic to the plane $\mathbb{R}^{2}$), which attracts all trajectories by trajectories lying on the manifold $\mathfrak{A}$. In particular, trajectories satisfy the Poincar\'{e}-Bendixson trichotomy. Note that the semiflow is dissipative due to the choice of the nonlinearity $g$. From this and since the equilibrium is unique, any point (except the equilibrium itself) from its unstable manifold (which lies in $\mathfrak{A}$ by definition; see \cite{Anikushin2020Geom}) must tend to a common periodic trajectory which encloses the equilibrium. Note that any point $\phi_{0}$ from the entire space, which is sufficiently close to the equilibrium, is exponentially attracted by a point $\phi^{*}_{0}$ from the unstable manifold. Moreover, $\phi^{*}_{0}$ coincides with the equilibrium only if $\phi_{0}$ belongs to the stable manifold. This finishes the proof.
\end{proof}
\begin{remark}
	Note that Proposition \ref{PROP: SSmodelHidden} does not guarantee that the periodic orbit will be orbitally stable\footnote{Note that a similar statement in \cite{Burkin2014Hidden} needs a clarification.}. However, there always exists at least one orbitally stable periodic orbit and any orbitally stable periodic orbit is asymptotically orbitally stable provided that it is isolated from other periodic orbits. In applications, we expect the periodic orbit from Proposition \ref{PROP: SSmodelHidden} to be the only periodic orbit and, consequently, to be asymptotically orbitally stable.
\end{remark}

\begin{figure}
	\begin{minipage}{.5\textwidth}
		\includegraphics[width=18pc,angle=0]{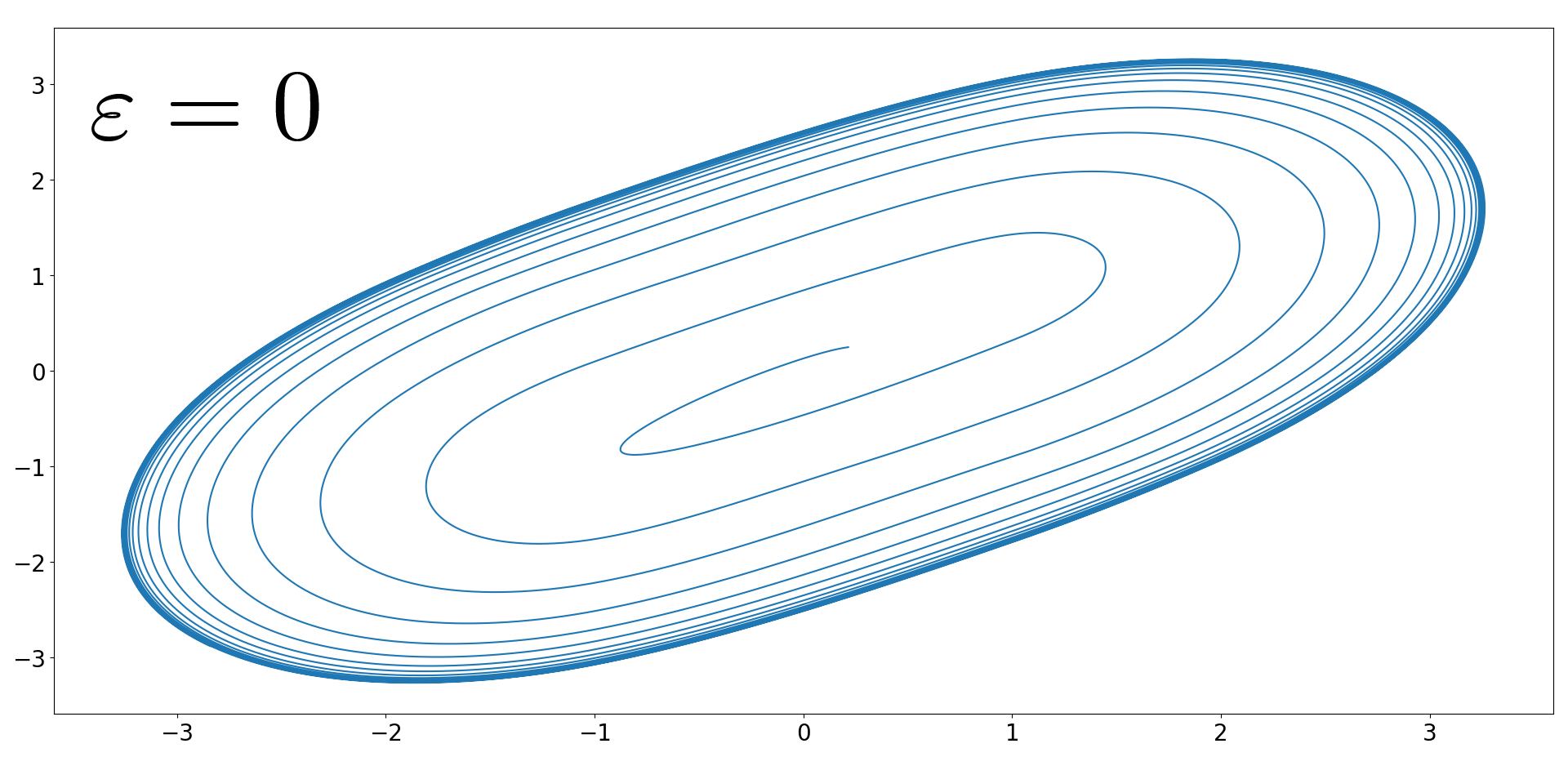}
		\includegraphics[width=18pc,angle=0]{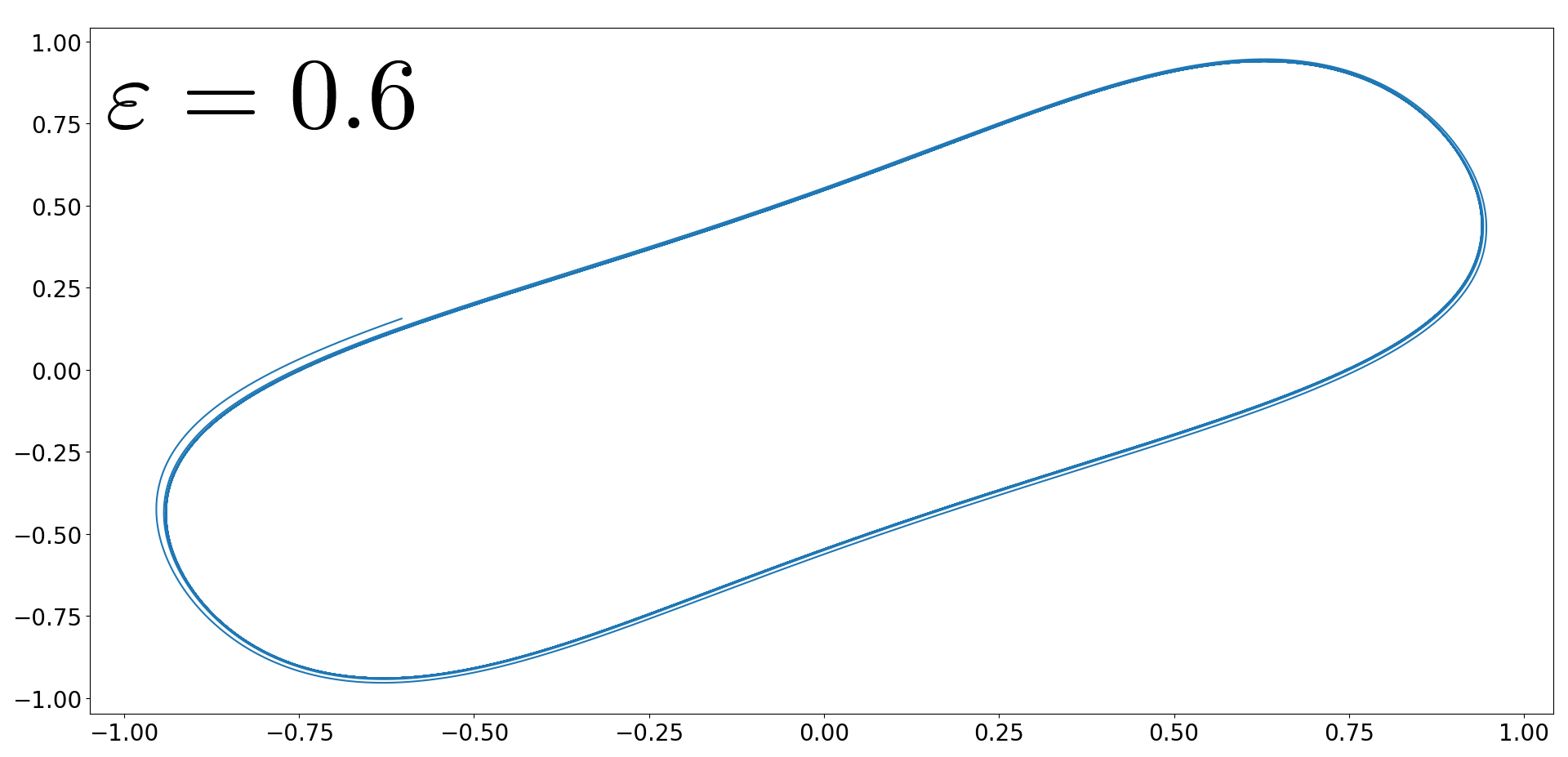}
	\end{minipage}%
	\begin{minipage}{.5\textwidth}
		\includegraphics[width=18pc,angle=0]{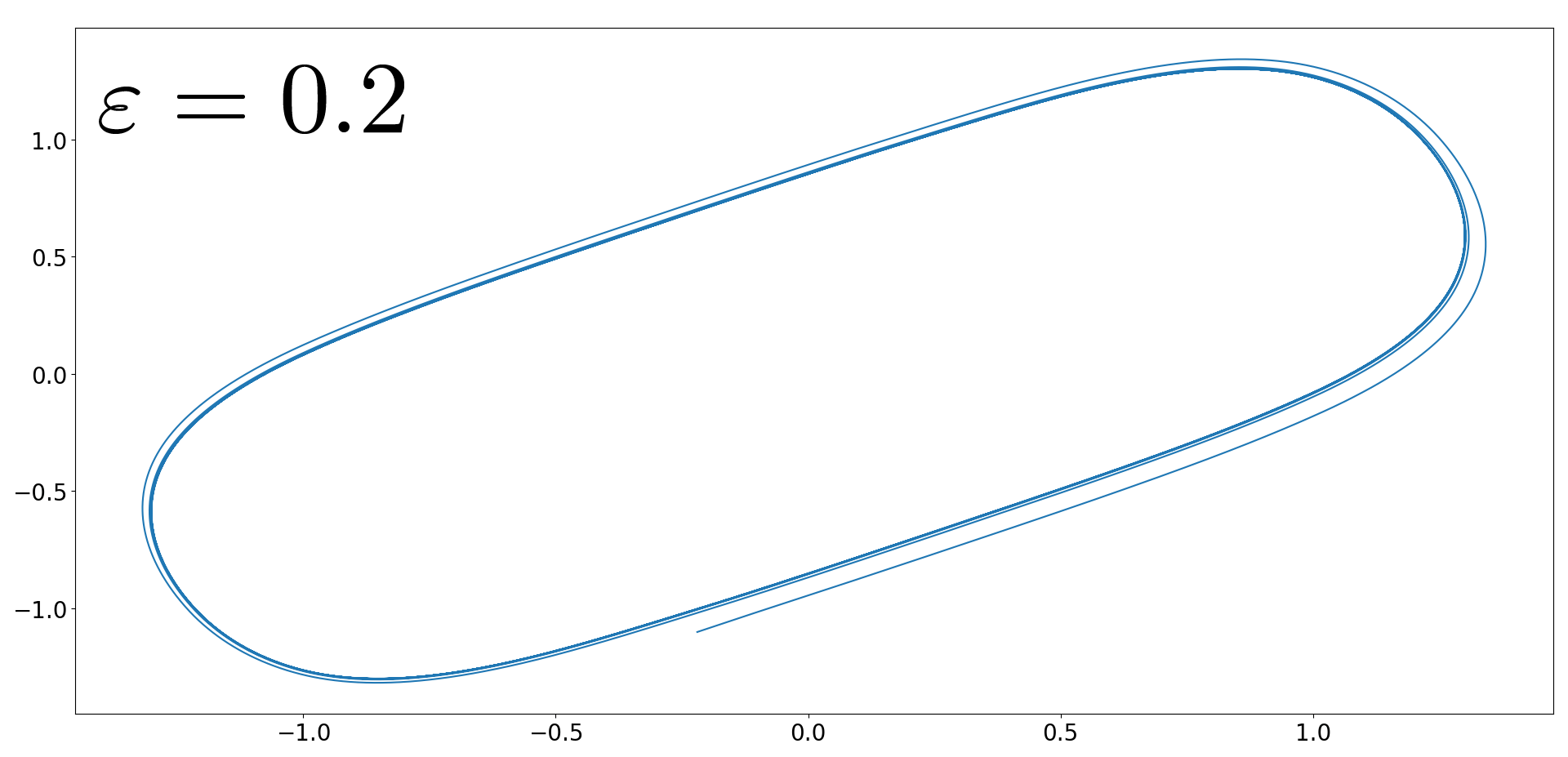}
		\includegraphics[width=18pc,angle=0]{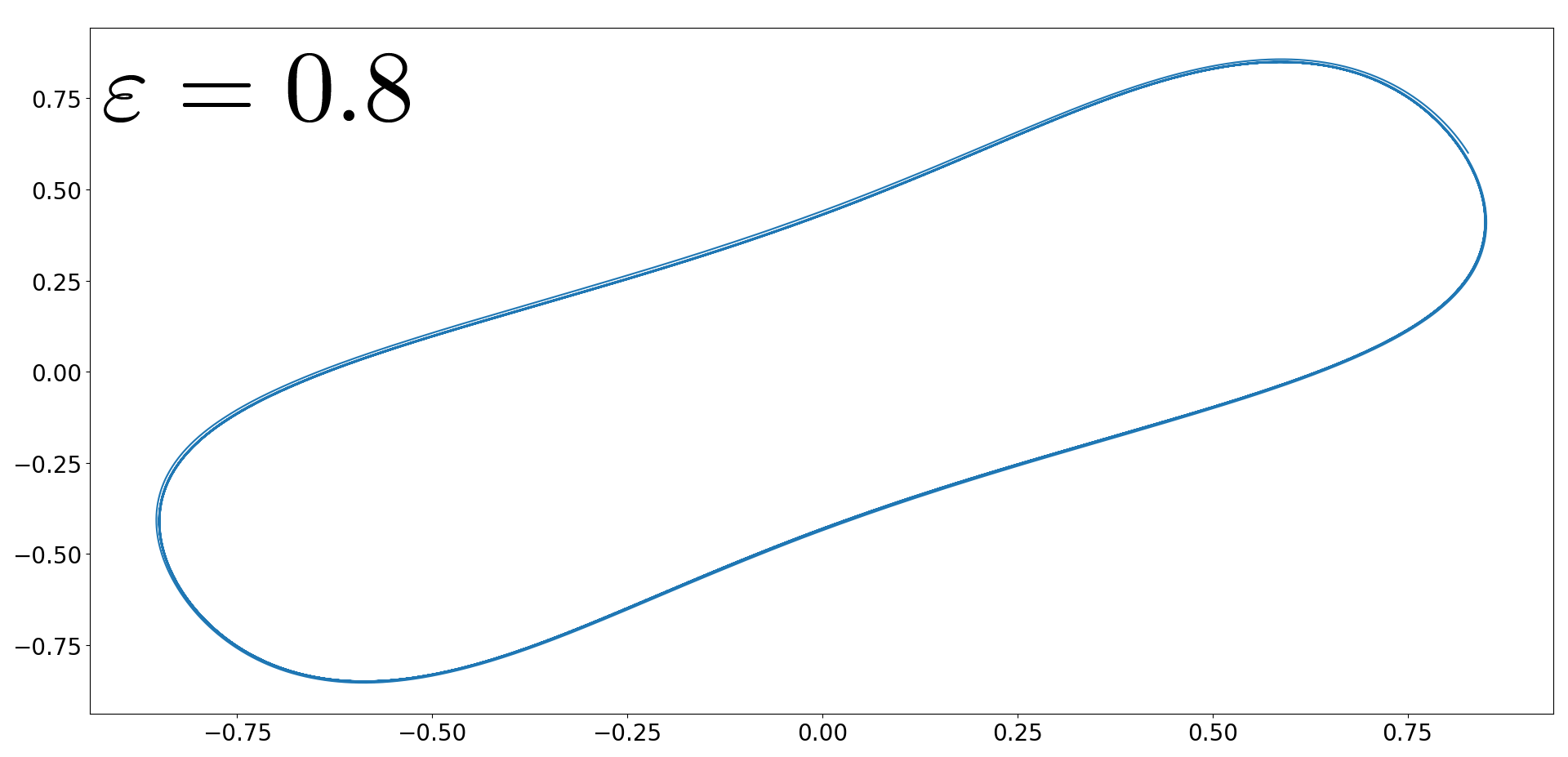}
	\end{minipage}%	
	\centering{
		\begin{minipage}{.5\textwidth}			
			\includegraphics[width=18pc,angle=0]{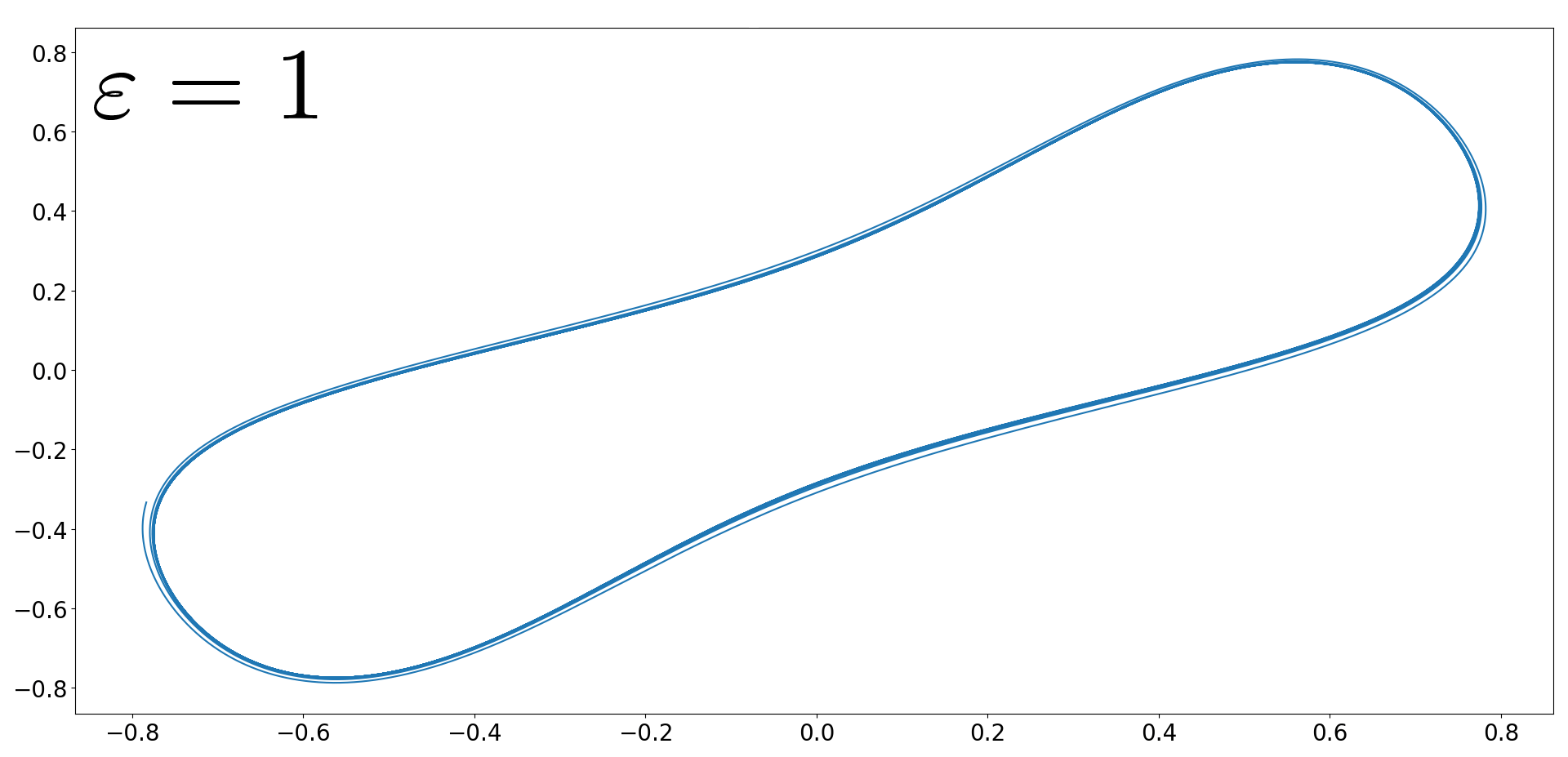}
		\end{minipage}
	}
	\caption{Results of the continuation by parameter procedure. At the initial step $\varepsilon=0$ we see the self-excited periodic orbit, the existence of which is guaranteed by Proposition \ref{PROP: SSmodelHidden}. At $\varepsilon=1$ this orbit evolves into a hidden periodic orbit of \eqref{EQ: ElNinoSSmodel}. All the trajectories are projected onto the $(\phi(-\tau),\phi(0))$ plane.}
	\label{Fig: ContParameterSSHid1}
\end{figure}

Fig. \ref{Fig: ContParameterSSHid1} shows some steps from the continuation by parameter procedure applied to \eqref{EQ: ElNinoSSmodelFamily} as $\varepsilon$ varies from $0$ to $1$. This leads to the discovery of a hidden periodic orbit of \eqref{EQ: ElNinoSSmodel}. Fig. \ref{FIG: SShidden1} justifies that it is indeed a \textit{hidden} periodic orbit, i.~e. it cannot be localized by taking initial data from a sufficiently small neighborhood of any equilibrium.
\section{Asynchronous oscillations in a ring array of coupled lossless transmission lines}
\label{SEC: AsyncOscillLLTL}
Let us consider the model for a ring array of coupled lossless transmission lines studied by J.~Wu and H.~Xia \cite{WuXiaSSRingArray1996}. It is described by a coupled system of $N$ neutral delay equations given by
\begin{equation}
	\label{EQ: RingArrayNDE}
	\begin{split}
		\frac{d}{dt} \left[ D(q) x^{k}_{t} \right] &= -a x^{k}(t) - bq x^{k}(t-\tau)-g(x^{k}(t))+qg(x^{k}(t-\tau)) + \\ &+ dD(q)\left[ x^{k+1}_{t} - 2 x^{k}_{t} + x^{k-1}_{t} \right], \ k=1,\ldots N \pmod N,
	\end{split}
\end{equation}
where $D(q)\phi := \phi(0) - q\phi(-\tau)$ and $q \in (0,1)$, $b$, $d$, $\tau$ are positive parameters. Note that we do not place restrictions on the sign of $a$. In the case $a>0$ there exists an increasing sequence of $q_{k} \in (0,1)$ such that for $q=q_{k}$ the trivial equilibrium of \eqref{EQ: RingArrayNDE} undergoes a Hopf bifurcation (see \cite{WuXiaSSRingArray1996}). The obtained periodic solutions are synchronous, i.~e. $x^{1}(t)=x^{2}(t)=\ldots=x^{N}(t)$ for all $t \in \mathbb{R}$. 

In the monograph \cite{GuoWu2013BifTh} S.~Guo and J.~Wu posed the problem of whether such a coupling can generate stable asynchronous periodic regimes (see p. 149 therein). Here we present certain parameters, although they may have no physical meaning, for which such regimes may occur as hidden or self-excited oscillations.

\begin{remark}
	For numerical integration of \eqref{EQ: RingArrayNDE}, we used the JiTCDDE package for Python (see G.~Ansmann \cite{AnsmannJITCODE2018}). We integrated \eqref{EQ: RingArrayNDE} on the time interval $[0,1500]$ with integration parameters $\operatorname{first\_step}=\operatorname{max\_step} = 10^{-4}$, $\operatorname{atol} = 10^{-5}$, $\operatorname{rtol} = 10^{-5}$.
\end{remark}

We consider \eqref{EQ: RingArrayNDE} with $g(x)=x^{3}$. Note that for $a=-1$, $d=0$ and small $q>0$ the neutral system \eqref{EQ: RingArrayNDE} is almost (up to the small term $q g(x^{k}(t-\tau))$) a system of $N$ uncoupled Suarez-Schopf oscillators \eqref{EQ: ElNinoSSmodel} with $\alpha = b q$. Below we will use parameters from the region $\Omega_{hid}$ (see Fig. \ref{FIG: SSHiddenCurves}) obtained for the Suarez-Schopf model \eqref{EQ: ElNinoSSmodel}. 

For example, let us take $N=2$, $\tau = 2.4$, $d=0.01$, $a=-1$, $b=60$ and $q=0.01$ (the point $(2.4,0.6)$ belong to the region $\Omega_{hid}$). Then there exist 4 asynchronous periodic regimes in the model (see Fig. \ref{fig: RCAAsynchHiddenAndSE} (Left) and Fig \ref{fig: AsyncGraphs}).

\begin{figure}[h!]
	\begin{minipage}{.5\textwidth}
		\includegraphics[width=18pc,angle=0]{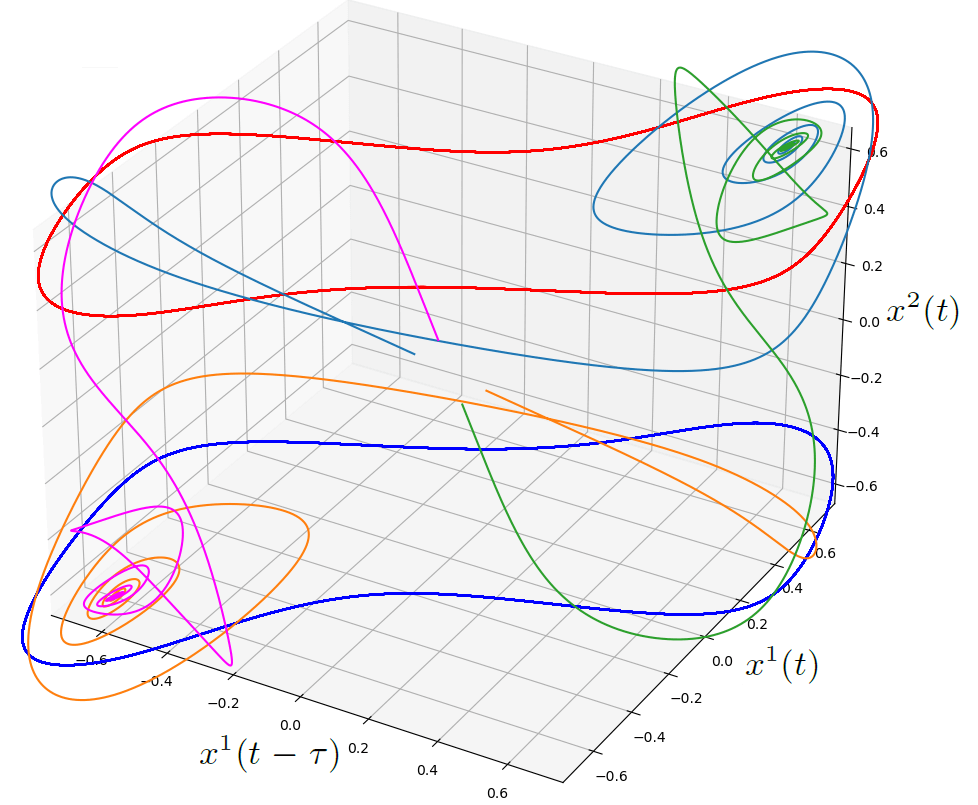}
	\end{minipage}%
	\begin{minipage}{.5\textwidth}
		\includegraphics[width=18pc,angle=0]{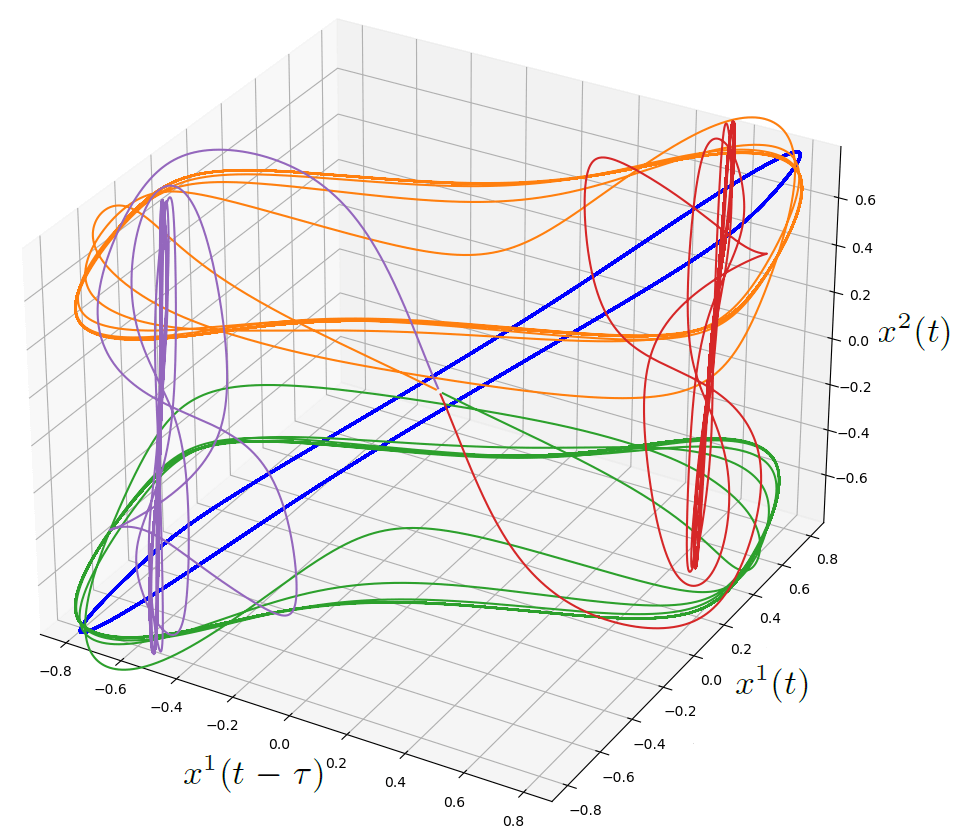}
	\end{minipage}%	
	\caption{(Left): Hidden asynchronous periodic orbits (red and blue) of \eqref{EQ: RingArrayNDE} discovered for the parameters $N=2$, $\tau = 2.4$, $a=-1$ $b=60$, $q = 0.01$ and $d=0.01$. Other trajectories start from a neighborhood of the trivial equilibrium. Note that there are also two more hidden orbits, which can be obtained by the symmetry $(x^{1},x^{2}) \mapsto (x^{2},x^{1})$. (Right): A hidden synchronous periodic orbit (blue) of \eqref{EQ: RingArrayNDE} coexists with four self-excited asynchronous periodic orbits localized by trajectories starting from a neighborhood of the trivial equilibrium (red, green, orange and purple). The model parameters are $N=2$,  $\tau = 2.45$, $a=-1$, $b=60$, $q = 0.01$ and $d=0.01$. All the trajectories are projected onto the $(x^{1}(t-\tau),x^{1}(t),x^{2}(t))$-space}
	\label{fig: RCAAsynchHiddenAndSE}
\end{figure}

We use the initial condition $(\phi_{1},\phi_{2})$, where $\phi_{1}(\theta) = 4 \cos(\theta) + 4$ and $\phi_{2}(\theta) = -3e^{\theta} + 3$, to localize the blue periodic orbit from Fig. \ref{fig: RCAAsynchHiddenAndSE} (Left). An intuition for the resulting behavior may be given as follows. For the chosen parameters, one may expect that both components $x^{1}$ and $x^{2}$ will start tending to the hidden periodic orbit of the Suarez-Schopf oscillator (as in Fig. \ref{FIG: SShidden1}) almost independently due to the smallness of the coupling coefficient $d$. However, near the hidden periodic orbit, the coupling may result in a dominant behavior of one component over the other, where the latter is pushed inside the unstable periodic orbit into the basin of attraction of a symmetric stationary state (see Fig. \ref{fig: SSPPunstable}). However, the dominated component cannot reach the equilibrium since it is coupled with the other one which causes a small periodic feedback on it. Thus, the dominated component tends to suffer small periodic phase-locked oscillations. This scenario is shown in Fig. \ref{fig: AsyncGraphs}.

Moreover, for $\tau=2.45$ we observe (see Fig. \ref{fig: RCAAsynchHiddenAndSE} (Right)) a single hidden synchronous periodic orbit (localized by the same as above initial data) and four self-excited asynchronous orbits.

However, the observed dominance of one component is not easy to achieve. One can vary $q$ or $d$ a bit, trying to continue the hidden orbits for the new parameters, and discover that usually both $x^{1}(\cdot)$ and $x^{2}(\cdot)$ will synchronously oscillate. This may indicate that the observed asynchronous behavior is not caused by the smallness of $q$ and $d$, but rather by their relation to other parameters. Thus, it is interesting to investigate continuation of the hidden orbits in the space of all parameters $(\tau,a,b,d,q)$ with the hope of moving towards more physical parameters.
\begin{figure}
	\centering
	\includegraphics[width=1\linewidth]{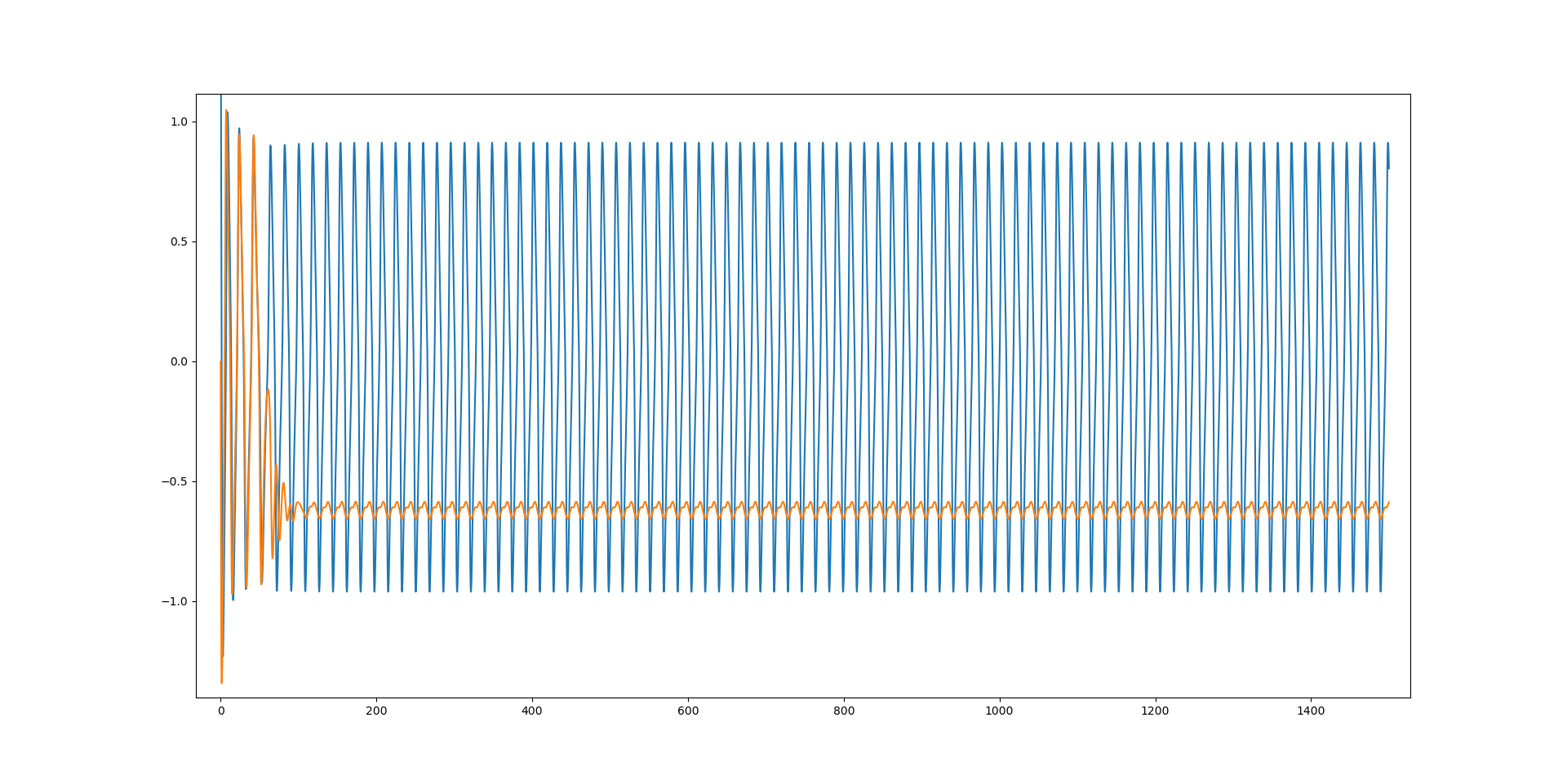}
	\caption{Graphs (versus time, horizontal) of the solution components (vertical) $x^{1}(t)$ (blue) and $x^{2}(t)$ (orange) which tend to the blue hidden periodic orbit from Fig. \ref{fig: RCAAsynchHiddenAndSE} (Left).}
	\label{fig: AsyncGraphs}
\end{figure}

\section*{Acknowledgments}
The authors are grateful to N.V.~Kuznetsov for many fruitful discussions on the topic and also to the anonymous referees of the first (rejected) version of the manuscript submitted to SIAM Journal on Applied Dynamical Systems.

\bibliographystyle{siamplain}
\bibliography{references}

\end{document}